\documentclass[10pt]{article}
\usepackage{mathtools}
\usepackage{graphicx}
\usepackage{amsmath,amsfonts,latexsym,amscd,amssymb,theorem}
\usepackage{graphicx}
\usepackage[ansinew]{inputenc}
\usepackage{epsfig}
\usepackage{amsmath}
\usepackage{amssymb}
\usepackage{afterpage}
\usepackage{selinput}
\usepackage{fancyhdr}
\newcommand{\ba}{\begin{eqnarray}}
\newcommand{\ea}{\end{eqnarray}}

\newtheorem{thm}{Theorem}[section]

\newtheorem{conjecture}{Conjecture}

\newtheorem{theorem}[thm]{Theorem}

\newtheorem{definition}[thm]{Definition}
\newtheorem{lemma}[thm]{Lemma}

\newtheorem{corollary}[thm]{Corollary}
\newtheorem{remark}[thm]{Remark}

\makeatletter
\newcommand*{\rom}[1]{\expandafter\@slowromancap\romannumeral #1@}
\makeatother

\makeatletter
\renewcommand\@biblabel[1]{#1.}
\makeatother
\begin{document}
\title{\textbf{Erd\"{o}s-Hajnal Conjecture for New Infinite Families of Tournaments}}
\maketitle


\begin{center}
\author{Soukaina Zayat \footnote{Department of Mathematics Faculty of Sciences I, Lebanese University, KALMA Laboratory, Beirut - Lebanon.\vspace{1.5mm} (soukaina.zayat.96@outlook.com)}, Salman Ghazal \footnote{Department of Mathematics Faculty of Sciences I, Lebanese University, Beirut - Lebanon.\\ (salman.ghazal@ul.edu.lb)}\footnote{Department of Mathematics and Physics, School of Arts and Sciences, Beirut International University, Beirut - Lebanon. \vspace{2mm} (salman.ghazal@liu.edu.lb)}}
\end{center}

\begin{abstract}
Erd\"{o}s-Hajnal conjecture states that for every undirected graph $H$ there exists $ \epsilon(H) > 0 $ such that every undirected graph on $ n $ vertices that does not contain $H$ as an induced subgraph contains a clique or a stable set of size at least $ n^{\epsilon(H)} $. This conjecture has a directed equivalent version stating that for every tournament $H$ there exists $ \epsilon(H) > 0 $ such that every $H$-free $n$-vertex tournament $T$ contains a transitive subtournament of order at least $ n^{\epsilon(H)} $. This conjecture is known to hold for a few infinite families of tournaments. In this paper we construct two new infinite families of tournaments $-$ the family of so-called galaxies with spiders and the family of so-called asterisms, and we prove the correctness of the conjecture for these two families. \end{abstract}

\section{Introduction}
\hspace{3.5mm} Let $ G $ be an undirected graph. We denote by $ V(G) $ the set of its vertices and by $ E(G) $ the set of its edges. We call $ \mid$$G$$\mid$ $:=$ $ \mid$$V(G)$$\mid$ the \textit{size} of $G$. A \textit{clique} in $G$ is a set of pairwise adjacent vertices and a \textit{stable set} in $G$ is a set of pairwise nonadjacent vertices. A \textit{digraph} $D$ is a pair $(V,E)$ of sets, such that $E\subset V \times V$, and for every $(x,y)\in E$ we must have $(y,x)\notin E$. In particular if $(x,y)\in E$, then $x \neq y$. $E$ is the arc set and $V$ is the vertex set and they are denoted by $E(D)$ and $V(D)$ respectively. We say that a digraph $D'$ is a \textit{subdigraph} of a digraph $D$ if $V(D') \subseteq V(D)$ and $E(D') \subseteq E(D)$. Let $X \subseteq V(D)$. The \textit{subdigraph of} $D$ \textit{induced by} $X$ is denoted by $D$$\mid$$X$, that is the digraph with vertex set $X$, such that for $x,y \in X$, $(x,y) \in E(D$$\mid$$ X)$ if and only if $(x,y) \in E(D)$. We say that $D$ \textit{contains} $D'$ if $D'$ is isomorphic to a subdigraph of $D$. A \textit{tournament} is a directed graph (digraph) such that for every pair $u$ and $v$ of vertices, exactly one of the arcs $(u,v)$ or $(v,u)$ exists. A tournament is \textit{transitive} if it contains no directed cycle. Let $T$ be a tournament. We write $\mid$$T$$\mid$ for $\mid$$V(T)$$\mid$ and we say that $\mid$$T$$\mid$ is the \textit{size} of $T$. If $(u,v)\in E(T)$, then we say that $u$ is \textit{adjacent to} $v$ (alternatively: $v$ is an \textit{outneighbor} of $u$), and we write $u\rightarrow v$. Also we say that $v$ is \textit{adjacent from} $u$ (alternatively: $u$ is an \textit{inneighbor} of $v$), and we write $v\leftarrow u$. For two sets of vertices $V_{1},V_{2}$ of $T$, we say that $V_{1}$ is \textit{complete to} (resp. \textit{from}) $V_{2}$ if every vertex of $V_{1}$ is adjacent to (resp. from) every vertex of $V_{2}$, and we write $V_{1} \rightarrow V_{2}$ (resp. $V_{1} \leftarrow V_{2}$). We say that a vertex $v$ is complete to (resp. from) a set $V$ if $\lbrace v \rbrace$ is complete to (resp. from) $V$, and we write $v \rightarrow V$ (resp. $v \leftarrow V$). Given a tournament $H$, we say that $T$ \textit{contains} $H$ if $H$ is isomorphic to $T$$\mid$$X$, for some $X \subseteq V(T)$. If $T$ does not contain $H$, we say that $T$ is $H$-$free$.  
\vspace{2.5mm}

Erd\"{o}s and Hajnal proposed the following conjecture (EHC) \cite{jhp}:
\begin{conjecture} For any undirected graph $H$ there exists $ \epsilon(H) > 0 $ such that every $n$-vertex undirected graph that does not contain $H$ as an induced subgraph contains a clique or a stable set of size at least $ n^{\epsilon(H)}. $
\end{conjecture}

In 2001 Alon et al. proved \cite{fdo} that Conjecture $1$ has an equivalent directed version, as follows:
\begin{conjecture} \label{a} For any tournament $H$ there exists $ \epsilon(H) > 0 $ such that every $ H$-free tournament with $n$ vertices contains a transitive subtournament of size at least $ n^{\epsilon(H)}. $
\end{conjecture}

A tournament $H$ \textit{satisfies the Erd\"{o}s-Hajnal Conjecture (EHC)} (equivalently: $H$ has the \textit{Erd\"{o}s-Hajnal property}) if there exists $ \epsilon(H) > 0 $ such that every $ H$-free tournament $T$ with $n$ vertices contains a transitive subtournament of size at least $ n^{\epsilon(H)}. $\vspace{2.5mm}

The Erd\"{o}s-Hajnal property is a \textit{hereditary property} \cite{ppp}, that is if a tournament $H$ has the Erd\"{o}s-Hajnal property, then all its subtournaments also have the Erd\"{o}s-Hajnal property.\vspace{3mm}

This paper is organized as follows:
\begin{itemize}
\setlength\itemsep{0em}
\item In section $2$ we give some definitions and preliminary lemmas, and we prove some lemmas needed in the proof of the main results in this paper.
\item In section $3$ we construct the infinite family of tournaments $-$ the family of so-called asterisms and prove that every asterism satisfies $EHC$. And as a consequence, every subtournament of an asterism satisfies $EHC$.
\item In section $4$ we construct the infinite family of tournaments $-$ the family of so-called galaxies with spiders and prove that every galaxy with spiders satisfies $EHC$.  
\item In section $5$ we define the ''\textit{merging operation}" for galaxies with spiders and asterisms, allowing us to build larger tournaments satisfying $EHC$ from smaller ones.
\end{itemize}
\section{Definitions and Preliminary Lemmas}
\hspace{3.5mm} In what follows, we denote by $[l]:=\{1,2,...,l\}$ for every positive integer $l$. Denote by $tr(T)$ the largest size of a transitive subtournament of a tournament $T$. For $X \subseteq V(T)$, write $tr(X)$ for $tr(T$$\mid$$X)$. Let $X, Y \subseteq V(T)$ be disjoint. Denote by $e_{X,Y}$ the number of directed arcs $(x,y)$, where $x \in X$ and $y \in Y$. The \textit{directed density from $X$ to} $Y$ is defined as $d(X,Y) = \frac{e_{X,Y}}{\mid X \mid .\mid Y \mid} $. We call $T$ $ \epsilon$-\textit{critical} for $ \epsilon > 0 $ if $tr(T) < $ $ \mid $$T$$ \mid^{\epsilon} $ but for every proper subtournament $S$ of $T$ we have: $tr(S) \geq $ $ \mid $$S$$ \mid^{\epsilon}. $
\begin{lemma} \cite{ml} \label{h}
Every tournament on $2^{k-1}$ vertices contains a transitive subtournament of size at least $k$.
 \end{lemma}
\begin{lemma} \cite{polll} \label{e} For every $N$ $ > 0 $, there exists $ \epsilon(N) > 0 $ such that for every $ 0 < \epsilon < \epsilon(N)$ every $ \epsilon $-critical tournament $T$ satisfies $ \mid $$T$$\mid$ $ \geq N$.
\end{lemma} 
\begin{lemma} \cite{polll} \label{v} Let $T$ be an $ \epsilon $-critical tournament with $\mid$$T$$\mid$ $ = n$ and $\epsilon$,$c > 0 $ be constants such that $ \epsilon < log_{\frac{c}{2}}(\frac{1}{2}). $ Then for every two disjoint subsets $X, Y \subseteq V(T)$ with $ \mid $$X$$ \mid$ $ \geq cn, \mid $$Y$$ \mid$ $ \geq cn $ there exist an integer $k \geq \frac{cn}{2} $ and vertices $ x_{1},...,x_{k} \in X $ and $ y_{1},...,y_{k} \in Y $ such that $ y_{i} $ is adjacent to $ x_{i} $ for $i = 1,...,k. $
\end{lemma}
\begin{lemma} \cite{polll} \label{f} Let $T$ be an $ \epsilon $-critical tournament with $\mid$$T$$\mid$ $ =n$ and $\epsilon$,$c,f > 0 $ be constants such that  $ \epsilon < log_{c}(1 - f)$. Then for every $A \subseteq V(T)$ with $ \mid $$A$$ \mid$ $ \geq cn$ and every transitive subtournament $G$ of $T$ with $ \mid $$G$$ \mid$ $\geq f.tr(T)$ and $V(G) \cap A = \phi$, we have: $A$ is not complete from $V(G)$ and $A$ is not complete to $V(G)$.
\end{lemma}
\begin{lemma} \label{s}
Let $f,c,\epsilon > 0$ be constants, where $0 <  f,c < 1$ and $0 < \epsilon < min\lbrace log_{\frac{c}{2}}(1-f), log_{\frac{c}{4}}(\frac{1}{2})\rbrace$. Let $T$ be an $ \epsilon $-critical tournament with $\mid$$T$$\mid$ $ =n$, and   let $S_{1},S_{2}$ be two disjoint transitive subtournaments of $T$ with $ \mid $$S_{1}$$ \mid$ $\geq f.tr(T)$ and $ \mid $$S_{2}$$ \mid$ $\geq f.tr(T)$. Let $A_{1},A_{2}$ be two disjoint subsets of $V(T)$ with $ \mid $$A_{1}$$ \mid$ $\geq cn$, $ \mid $$A_{2}$$ \mid$ $\geq cn$, and $A_{1},A_{2} \subseteq V(T) \backslash (V(S_{1})\cup V(S_{2}))$. Then there exist vertices $a,x,s_{1},s_{2}$ such that $a\in A_{1}, x\in A_{2}, s_{1}\in S_{1}, s_{2}\in S_{2}$, $\lbrace a,s_{2}\rbrace \leftarrow x$, and $s_{1}\leftarrow a$. Similarly there exist vertices $d_{1},d_{2},u_{1},u_{2}$ such that $d_{1}\in A_{1}, d_{2}\in A_{2}, u_{1}\in S_{1}, u_{2}\in S_{2}$, $d_{1}\leftarrow\lbrace u_{1},d_{2}\rbrace$, and $d_{2}\leftarrow u_{2}$.
\end{lemma}
\begin{proof}
We will prove only the first statement  because the latter can be proved analogously. Let $A_{1}^{*} = \lbrace a \in A_{1}: \exists$ $ s \in S_{1}$ and $s \leftarrow a \rbrace$ and let $A_{2}^{*} = \lbrace x \in A_{2}: \exists$ $ v \in S_{2}$ and $v \leftarrow x \rbrace$. Then $A_{1}\backslash A_{1}^{*}$ is complete from $S_{1}$ and $A_{2}\backslash A_{2}^{*}$ is complete from $S_{2}$. Now assume that $\mid$$A_{1}^{*}$$\mid$ $< \frac{\mid A_{1} \mid}{2}$, then $\mid$$A_{1}\backslash A_{1}^{*}$$\mid$ $\geq \frac{\mid A_{1} \mid}{2} \geq \frac{c}{2}n$. Since $\mid$$ S_{1}$$\mid$ $\geq f.tr(T)$ and since $\epsilon < log_{\frac{c}{2}}(1-f)$, then Lemma \ref{f} implies that $A_{1}\backslash A_{1}^{*}$ is not complete from $S_{1}$, a contradiction. Then $\mid$$A_{1}^{*}$$\mid$ $\geq \frac{\mid A_{1} \mid}{2} \geq \frac{c}{2}n$. Similarly we prove that $\mid$$A_{2}^{*}$$\mid$ $\geq  \frac{c}{2}n$. Now since $\epsilon < log_{\frac{c}{4}}(\frac{1}{2})$, then Lemma \ref{v} implies that there exists $ k \geq \frac{c}{4}n$, there exist vertices $ a_{1},...,a_{k} \in A_{1}^{*}$, and there exist vertices $ x_{1},...,x_{k} \in A_{2}^{*}$, such that $a_{i} \leftarrow x_{i}$ for $i = 1,...,k$.  So, there exist $ a_{1} \in A_{1}^{*}$, $ x_{1} \in A_{2}^{*}$, $ s_{1} \in S_{1}$, and $ s_{2} \in S_{2}$, such that $\lbrace a_{1},s_{2}\rbrace \leftarrow x_{1}$ and $s_{1}\leftarrow a_{1}$. $\hfill{\square}$ 
\end{proof}
\begin{lemma}\label{r}
Let $f_{1},...,f_{m},c,\epsilon > 0$ be constants, where $0 <  f_{1},...,f_{m},c < 1$ and $0 < \epsilon < log_{\frac{c}{2m}}(1-f_{i})$ for $i=1,...,m$. Let $T$ be an $ \epsilon $-critical tournament with $\mid$$T$$\mid$ $ =n$, and   let $S_{1},...,S_{m}$ be m disjoint transitive subtournaments of $T$ with $ \mid $$S_{i}$$ \mid$ $\geq f_{i}.tr(T)$ for $i=1,...,m$. Let $A \subseteq V(T) \backslash (\bigcup_{i=1}^{m} V(S_{i}))$ with $ \mid $$A$$ \mid$ $ \geq cn$. Then there exist vertices $s_{1},...,s_{m},a$ such that $a\in A$, $s_{i}\in S_{i}$ for $i=1,...,m$, and $\lbrace s_{1},...,s_{m-1} \rbrace \leftarrow \lbrace a \rbrace \leftarrow \lbrace s_{m} \rbrace$. Similarly there exist vertices $u_{1},...,u_{m},b$ such that $b\in A$, $u_{i}\in S_{i}$ for $i=1,...,m$, and $\lbrace s_{1} \rbrace \leftarrow \lbrace a \rbrace \leftarrow \lbrace s_{2},...,s_{m} \rbrace $. Similarly there exist vertices $p_{1},...,p_{m},g$ such that $g\in A$, $p_{i}\in S_{i}$ for $i=1,...,m$, and $\lbrace g \rbrace$ is complete to $\lbrace p_{1},...,p_{m} \rbrace$. Similarly there exist vertices $q_{1},...,q_{m},v$ such that $v\in A$, $q_{i}\in S_{i}$ for $i=1,...,m$, and $\lbrace v \rbrace$ is complete from $\lbrace q_{1},...,q_{m} \rbrace$.
\end{lemma}
\begin{proof}
We will prove only the first statement  because the rest can be proved analogously.
Let $A_{i} \subseteq A$ such that $A_{i}$ is complete from $S_{i}$ for $i = 1,...,m-1$ and complete to $S_{i}$ for $i=m$. Let $j\in [m]$. If $\mid$$A_{j}$$\mid $ $\geq \frac{\mid A \mid}{2m}\geq \frac{c}{2m}n$, then this will contradicts Lemma \ref{f} since $\mid$$S_{j}$$\mid$ $ \geq f_{j}tr(T)$ and $\epsilon < log_{\frac{c}{2m}}(1-f_{j})$. Then for all $ i \in [m]$, $\mid$$A_{i}$$\mid$ $ < \frac{\mid A \mid}{2m}$. Let $A^{*} = A\backslash (\bigcup_{i=1}^{m}A_{i})$, then $\mid$$A^{*}$$\mid$ $ > $ $\mid $$A$$ \mid-m.\frac{\mid A \mid}{2m} \geq \frac{\mid A \mid}{2}$. Then $A^{*} \neq \phi$. Fix $a\in A^{*}$. So there exist vertices $s_{1},...,s_{m}$, such that $s_{i}\in S_{i}$ for $i=1,...,m$, and $\lbrace s_{1},...,s_{m-1} \rbrace \leftarrow \lbrace a \rbrace \leftarrow \lbrace s_{m} \rbrace$.  $\hfill{\square}$ 
\end{proof}
\begin{lemma} \cite{polll} \label{b} Let $A_{1},A_{2}$ be two disjoint sets such that $d(A_{1},A_{2}) \geq 1-\lambda$ and let $0 < \eta_{1},\eta_{2} \leq 1$. Let $\widehat{\lambda} = \frac{\lambda}{\eta_{1}\eta_{2}}$. Let $X \subseteq A_{1}, Y \subseteq A_{2}$ be such that $\mid$$X$$\mid$ $\geq \eta_{1} \mid$$A_{1}$$\mid$ and $\mid$$Y$$\mid$ $\geq \eta_{2} \mid$$A_{2}$$\mid$. Then $d(X,Y) \geq 1-\widehat{\lambda}$. 
\end{lemma}

The following is introduced in \cite{bnmm}.\\
Let $ c > 0, 0 < \lambda < 1 $ be constants, and let $w$ be a $ \lbrace 0,1 \rbrace$-vector of length $ \mid $$w$$ \mid $. Let $T$ be a tournament with $ \mid $$T$$ \mid$ $ = n. $ A sequence of disjoint subsets $ \chi = (S_{1}, S_{2},..., S_{\mid w \mid}) $ of $V(T)$ is a \textit{smooth $ (c,\lambda, w)$-structure} if:\\
$\bullet$ whenever $ w_{i} = 0 $ we have $ \mid $$S_{i}$$ \mid$ $ \geq cn $ (we say that $ S_{i} $ is a \textit{linear set}).\\
$\bullet$ whenever $ w_{i} = 1 $ the tournament $T$$\mid$$ S_{i} $ is transitive and $ \mid $$S_{i}$$ \mid$ $ \geq c.tr(T) $ (we say that $ S_{i} $ is a \textit{transitive set}).\\
$\bullet$ $ d(\lbrace v \rbrace, S_{j}) \geq 1 - \lambda $ for $v \in S_{i} $ and $ d(S_{i}, \lbrace v \rbrace) \geq 1 - \lambda $ for $v \in S_{j}, i < j $ (we say that $\chi$ is \textit{smooth}).
\begin{theorem} \cite{bnmm} \label{i}
Let $S$ be a tournament, let $w$ be a $ \lbrace 0,1 \rbrace $-vector, and let $ 0 < \lambda_{0} < \frac{1}{2} $ be a constant. Then there exist $ \epsilon_{0}, c_{0} > 0 $ such that for every $ 0 < \epsilon < \epsilon_{0} $, every $ S$-free $ \epsilon $-critical tournament contains a smooth $ (c_{0}, \lambda_{0},w)$-structure.
\end{theorem}
\begin{corollary}\cite{sss} \label{c}
Let $S$ be a tournament, let $w$ be a $ \lbrace 0,1 \rbrace $-vector, and let $m \in \mathbb{N}^{*}$, $ 0 < \lambda_{1} < \frac{1}{2} $ be constants. Then there exist $ \epsilon_{1}, c_{1} > 0 $ such that for every $ 0 < \epsilon < \epsilon_{1} $, every $ S$-free $ \epsilon $-critical tournament contains a smooth $ (c_{1}, \lambda_{1},w)$-structure $(S_{1},...,S_{\mid w \mid})$ such that for all $ 1 \leq j \leq$ $ \mid$$ w$$ \mid$, $\mid$$S_{j}$$\mid$ is divisible by $m$.
\end{corollary}
\begin{proof}
Let $\lambda_{0} = \frac{\lambda_{1}}{2}$. By Theorem \ref{i}, there exist $ \epsilon_{0}, c_{0} > 0 $ such that for every $ 0 < \epsilon < \epsilon_{0} $, every $ S$-free $ \epsilon $-critical tournament contains a smooth $ (c_{0}, \lambda_{0},w)$-structure. Denote this structure by $(A_{1},...,A_{\mid w \mid})$. Now for all $j\in $ $ [\mid$$ w$$ \mid ]$, take an arbitrary subset $S_{j}$ of $A_{j}$, such that $\mid$$S_{j}$$\mid$ is divisible by $m$ and $\mid$$S_{j}$$\mid$ is of maximum size. Notice that for all $ j \in$ $[\mid$$ w$$ \mid ]$, $\mid$$S_{j}$$\mid$ $\geq$ $ \mid$$A_{j}$$\mid -$ $ m$. Taking $\epsilon_{0}$ small enough we may assume that $tr(T) \geq \frac{2m}{c_{0}}$ by Lemmas \ref{h} and \ref{e}. For all $j \in$ $ [\mid$$ w$$ \mid ]$ with $w(j) = 1$, we have: $\mid$$A_{j}$$\mid$ $ \geq c_{0}tr(T) \geq c_{0}\frac{2m}{c_{0}} = 2m$. Then $\mid$$S_{j}$$\mid$ $\geq$ $ \mid$$A_{j}$$\mid -$ $ m \geq$ $ \mid$$A_{j}$$\mid - $ $\frac{\mid A_{j}\mid}{2} = \frac{\mid A_{j}\mid}{2}$. And for all $j \in$ $ [\mid$$ w$$ \mid ]$ with $w(j) = 0$, we have: $\mid$$A_{j}$$\mid$ $ \geq c_{0}n \geq c_{0}tr(T) \geq 2m$. Then $\mid$$S_{j}$$\mid$ $\geq$ $ \frac{\mid A_{j}\mid}{2}$.
Now Lemma \ref{b} implies that $(S_{1},...,S_{\mid w \mid})$ is a smooth $(\frac{c_{0}}{2}, 2\lambda_{0},w)$-structure of $T$. Thus taking $c_{1} = \frac{c_{0}}{2}$, completes the proof. $\hfill{\square}$  
\end{proof}\vspace{3mm} 

Let $(S_{1},...,S_{\mid w \mid})$ be a smooth $(c,\lambda ,w)$-structure of a tournament $T$, let $i \in \lbrace 1,...,\mid$$w$$\mid \rbrace$, and let $v \in S_{i}$. For $j\in \lbrace 1,2,...,\mid$$w$$\mid \rbrace \backslash \lbrace i \rbrace$, denote by $S_{j,v}$ the set of the vertices of $S_{j}$ adjacent from $v$ for $j > i$ and adjacent to $v$ for $j<i$.
\begin{lemma} \label{g}\cite{sss} Let $0<\lambda<1$, $0<\gamma \leq 1$ be constants and let $w$ be a $\lbrace 0,1 \rbrace$$-$vector. Let $(S_{1},...,S_{\mid w \mid})$ be a smooth $(c,\lambda ,w)$-structure of a tournament $T$ for some $c>0$. Let $j\in \lbrace 1,...,\mid$$w$$\mid \rbrace$. Let $S_{j}^{*}\subseteq S_{j}$ such that $\mid$$S_{j}^{*}$$\mid$ $\geq \gamma \mid$$S_{j}$$\mid$, and let $A= \lbrace x_{1},...,x_{k} \rbrace \subseteq \displaystyle{\bigcup_{i\neq j}S_{i}}$ for some positive integer $k$. Then $\mid$$\displaystyle{\bigcap_{x\in A}S^{*}_{j,x}}$$\mid$ $\geq (1-k\frac{\lambda}{\gamma})\mid$$S_{j}^{*}$$\mid$. In particular $\mid$$\displaystyle{\bigcap_{x\in A}S_{j,x}}$$\mid$ $\geq (1-k\lambda)\mid$$S_{j}$$\mid$.
\end{lemma}
\begin{proof}
The proof is by induction on $k$. without loss of generality assume that $x_{1} \in S_{i}$ and $j<i$. Since $\mid$$S_{j}^{*}$$\mid$ $\geq \gamma \mid$$S_{j}$$\mid$, then by Lemma \ref{b}, $d(S^{*}_{j},\lbrace x_{1}\rbrace) \geq 1-\frac{\lambda}{\gamma}$. So $1-\frac{\lambda}{\gamma} \leq d(S^{*}_{j},\lbrace x_{1}\rbrace) = \frac{\mid S^{*}_{j,x_{1}}\mid}{\mid S_{j}^{*}\mid}$. Then $\mid$$S^{*}_{j,x_{1}}$$\mid$ $\geq (1-\frac{\lambda}{\gamma})$$\mid$$S_{j}^{*}$$\mid$ and so true for $k=1$.
Suppose the statement is true for $k-1$.\\ $\mid$$\displaystyle{\bigcap_{x\in A}S^{*}_{j,x}}$$\mid$ $=\mid$$(\displaystyle{\bigcap_{x\in A\backslash \lbrace x_{1}\rbrace}S^{*}_{j,x}})\cap S^{*}_{j,x_{1}}$$\mid$ $= \mid$$\displaystyle{\bigcap_{x\in A\backslash \lbrace x_{1}\rbrace}S^{*}_{j,x}}$$\mid$ $+$ $\mid$$S^{*}_{j,x_{1}}$$\mid$ $- \mid$$(\displaystyle{\bigcap_{x\in A\backslash \lbrace x_{1}\rbrace}S^{*}_{j,x}})\cup S^{*}_{j,x_{1}}$$\mid$ $\geq (1-(k-1)\frac{\lambda}{\gamma})\mid$$S_{j}^{*}$$\mid$ $+$ $(1-\frac{\lambda}{\gamma})\mid$$S_{j}^{*}$$\mid$ $-$ $\mid$$S_{j}^{*}$$\mid$ $= (1-k\frac{\lambda}{\gamma})\mid$$S_{j}^{*}$$\mid$. $\hfill{\square}$       
\end{proof}\vspace{2mm}\\

Let $ \theta = (v_{1},...,v_{n}) $ be an ordering of the vertex set $V(D)$ of an $ n- $vertex digraph $D$.
An arc $ (v_{i},v_{j})\in E(D) $ is a \textit{backward arc of $D$ under} $ \theta $ if $ i > j $. For any two vertices $v_{i}, v_{j} $ with $i\neq j$, we say that $ v_{i} $ is \textit{before} $ v_{j} $ if $i<j$ and \textit{after} $ v_{j} $ otherwise. We say that a vertex $ v_{j} $ is \textit{between} two vertices $ v_{i},v_{k} $ under $ \theta = (v_{1},...,v_{n}) $ if $ i < j < k $ or $ k < j < i $. For a subset $N\subseteq V(D)$, we say that $v_{i}\in N$ is a \textit{left point of $N$ under $\theta$} if $i=$ min$\lbrace j: v_{j}\in N\rbrace$. And we say that $v_{i}\in N$ is a \textit{right point of $N$ under $\theta$} if $i=$ max$\lbrace j: v_{j}\in N\rbrace$.
 The graph of backward arcs under $ \theta $, denoted by $ B(D,\theta) $, is the undirected graph that has vertex set $V(D)$, and $ v_{i}v_{j} \in E(B(D,\theta)) $ if and only if $ (v_{i},v_{j}) $ or $ (v_{j},v_{i}) $ is a backward arc of $D$ under $ \theta $. A tournament $S$ on $n$ vertices with $V(S)= \lbrace u_{1},u_{2},...,u_{p}\rbrace$ is a \textit{right star} (resp. \textit{left star}) (resp. \textit{middle star}) if there exists an ordering $\theta^{*} = (u_{1},u_{2},...,u_{p})$ of its vertices such that the backward arcs of $S$ under $\theta^{*}$ are $(u_{p},u_{i})$ for all $i\in [p-1]$ (resp. $(u_{i},u_{1})$ for all $i\geq 2$) (resp. $(u_{i},u_{r})$ for all $i\geq r$ and $(u_{r},u_{i})$ for all $i\in [r-1]$, where $2\leq r\leq p-1$). In this case we write $S = \lbrace u_{1},u_{2},...,u_{p}\rbrace$ and we call $\theta^{*} = (u_{1},u_{2},...,u_{p})$ a \textit{right star ordering} (resp. \textit{left star ordering}) (resp. \textit{middle star ordering}) of $S$, $u_{p}$ (resp. $u_{1}$) (resp. $u_{r}$) the \textit{center of} $S$, and $u_{1},...,u_{p-1}$ (resp. $u_{2},...,u_{p}$) (resp. $u_{1},...,u_{r-1},u_{r+1},...,u_{p}$) the \textit{leaves of} $S$. A \textit{star} is a left star or a right star (note that a star is not a middle star, a star is either left or right). A \textit{star ordering} is a left star ordering or a right star ordering. Note that in the case $p=2$ we may choose arbitrarily any one of the two vertices to be the center of the star, and the other vertex is then considered to be the leaf.
   A \textit{star} $S=\lbrace v_{i_{1}},...,v_{i_{t}}\rbrace$ \textit{of $D$ under $\theta$} (where $i_{1}<...<i_{t}$) is the subdigraph of $D$ induced by $\lbrace v_{i_{1}},...,v_{i_{t}}\rbrace$ such that $S$ is a star and $S$ has the star ordering $ (v_{i_{1}},...,v_{i_{t}})$ under $\theta$ (i.e $(v_{i_{1}},...,v_{i_{t}})$ is the restriction of $\theta$ to $V(S)$ and $ (v_{i_{1}},...,v_{i_{t}})$ is a star ordering of $S$). A \textit{middle star} $S=\lbrace v_{i_{1}},...,v_{i_{t}}\rbrace$ \textit{of $D$ under $\theta$} (where $i_{1}<...<i_{t}$) is the subdigraph of $D$ induced by $\lbrace v_{i_{1}},...,v_{i_{t}}\rbrace$ such that $S$ is a middle star and $S$ has the middle star ordering $ (v_{i_{1}},...,v_{i_{t}})$ under $\theta$.\vspace{2.5mm}
   
   A tournament $T$ is a \textit{nebula} if there exists an ordering $\theta$ of its vertices such that $V(T)$ is the disjoint union of $V(Q_{1}),...,V(Q_{l}),X$ where $Q_{i}$ is either a star or a middle star of $T$ under $\theta$ for $i=1,...,l$, and for every $x\in X$, $\lbrace x \rbrace$ is a singleton component of $B(T,\theta)$  (note that there is no condition concerning the location of the centers of the stars and middle stars). We also say that $T$ \textit{is a \textit{nabula} under $\theta$}, and $\theta$ is called a \textit{nebula ordering of $T$}. 

A tournament $T$ is a \textit{galaxy} if there exists an ordering $\theta$ of its vertices such that $V(T)$ is the disjoint union of $V(S_{1}),...,V(S_{l}),X$ where $S_{1},...,S_{l}$ are the stars of $T$ under $\theta$, and for every $x\in X$, $\lbrace x \rbrace$ is a singleton component of $B(T,\theta)$, and no center of a star is between leaves of another star under $\theta$. We also say that $T$ \textit{is a \textit{galaxy} under $\theta$}. If $X=\phi$, we say that $T$ is a \textit{regular galaxy}.\vspace{2.5mm} 

 The condition concerning the type of the stars in galaxies (that is middle stars are not allowed) and that no center of a star appears in the ordering between leaves of another star is necessary to make it possible looking for centers of the stars in linear sets and looking for leaves of the same star in the same transitive set, and so making the proof of the following theorem works.
\begin{theorem}\cite{polll}
Every galaxy satisfies the Erd\"{o}s-Hajnal conjecture.
\end{theorem}
\section{Asterisms} \label{q}
\hspace{3.5mm} There exist infinitely many tournaments with no nebula ordering. That motivates us to work on new configuration of backward arcs. In this paper we studied the structure of the graph of backward arcs under several orderings of the same tournament, and we introduce the technique of \textit{"Corresponding Digraphs"} to prove the conjecture for a new infinite family of tournaments $-$ the family of so-called asterisms. First we will define formally the class of asterisms and the notion of corresponding  digraphs. Then we will give an overview that explains main steps of the proof, why the known ideas failed to succeed in this class, and how the new technique works. And finally we will prove the conjecture for asterisms. \\Before defining formally the class of asterisms, we would like to describe roughly this infinite class of tournaments. In asterisms we treat the existence of new substructures called $\beta$-asteroids (to be defined below), and we allow for the centers of stars to appear between the leaves of another stars on three vertices under the condition that these stars on three vertices must be subtournaments of $\beta$-asteroids (see Figure \ref{fig:generalizedasteroid}). What is special about asterisms is that infinitely many asterisms  has no nebula ordering (see Section \ref{secgalaxywithspiders} for the definition of nebula ordering). Instead every asterism has a special ordering for which  new configuration of backward arcs appears under this ordering (different than stars), and in some portions, centers of stars are allowed  to appear between leaves of another stars.     
\subsection{Definitions and tools}
\hspace{3.5mm} The family of asterisms contains infinitely many tournaments that are neither galaxies nor constellations (constellations are defined in \cite{kg}). In asterisms we deal with new substructures called asteroids, that are different than stars and middle stars. We will start by explaining how we construct the first infinite family of tournaments $-$ the family of so-called asterisms.\\
An \textit{asteroid} $\mathcal{A} = \lbrace 1,2,3,4,5 \rbrace$ is a five-vertex tournament with $V(\mathcal{A}) = \lbrace 1,2,3,4,5 \rbrace$ and $E(\mathcal{A}) = \lbrace (1,2),(1,3),\\(4,1),(1,5),(2,3),(2,4),(2,5),(3,4),(5,3),(5,4) \rbrace$. \\
The following are four special tournaments obtained from $\mathcal{A}$.  The \textit{left $\beta_{1}$-asteroid} is the tournament obtained from $\mathcal{A}$ by adding two extra vertices $6$ and $7$ and making $6$ adjacent from $2$ and $7$, and making $7$ adjacent to $5$ and $6$. The \textit{right $\beta_{1}$-asteroid} is the tournament obtained from $\mathcal{A}$ by adding two extra vertices $6$ and $7$, and making $6$ adjacent to $3$ and $7$, and making $7$ adjacent from $5$ and $6$. The \textit{left $\beta_{2}$-asteroid} is the tournament obtained from $\mathcal{A}$ by adding two extra vertices $6$ and $7$ and making $6$ adjacent from $2$ and $7$, and making $7$ adjacent to $6$. The \textit{right $\beta_{2}$-asteroid} is the tournament obtained from $\mathcal{A}$ by adding two extra vertices $6$ and $7$, and making $6$ adjacent to $3$ and $7$, and making $7$ adjacent from $6$.
 In Figure \ref{fig:generalizedasteroid} we define some crucial orderings of the vertex set of the above four tournaments.
 \begin{figure}[h]
	\centering
	\includegraphics[width=1\linewidth]{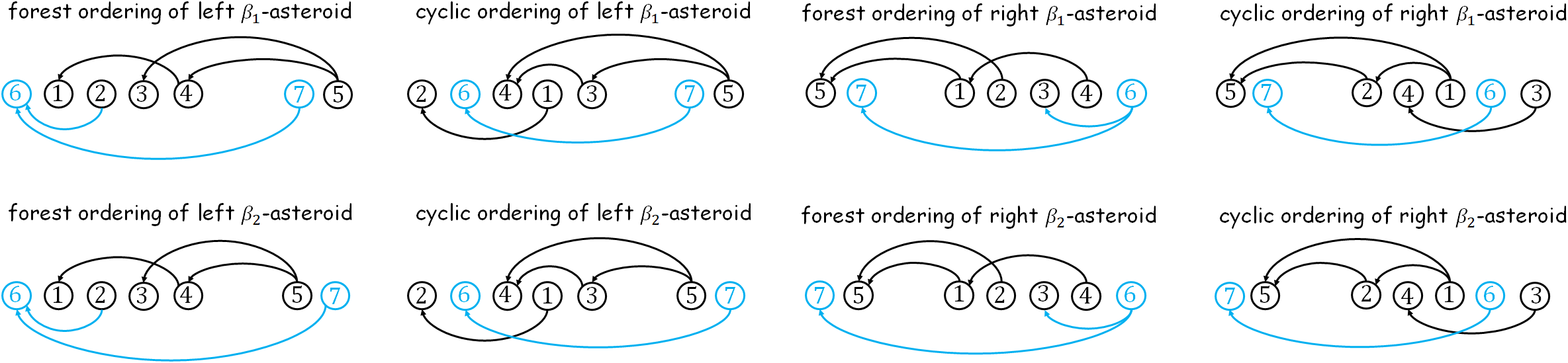}
	\caption{Crucial orderings of the vertices of left and right $\beta_{1}$-asteroid, left and right $\beta_{2}$-asteroid. All the nondrawn arcs are forward.}
	\label{fig:generalizedasteroid}
\end{figure}

Let $ \theta = (v_{1},...,v_{n}) $ be an ordering of the vertex set $V(T)$ of an $ n$-vertex tournament $T$. A \textit{left $\beta_{1}$-asteroid} (resp. \textit{left $\beta_{2}$-asteroid}) $\mathcal{A}^{\beta} := \lbrace v_{i_{1}},v_{i_{2}},v_{i_{3}},v_{i_{4}},v_{i_{5}},v_{i_{6}},v_{i_{7}} \rbrace$ \textit{of $T$ under} $\theta$ (where $i_{1}<...<i_{7}$) is an induced subtournament of $T$ with vertex set $\lbrace v_{i_{1}},v_{i_{2}},v_{i_{3}},v_{i_{4}},v_{i_{5}},v_{i_{6}},v_{i_{7}} \rbrace$, such that $\mathcal{A}^{\beta}$ is a left $\beta_{1}$-asteroid (resp. left $\beta_{2}$-asteroid) and has the forest ordering $(v_{i_{1}},v_{i_{2}},v_{i_{3}},v_{i_{4}},v_{i_{5}},v_{i_{6}},v_{i_{7}})$ under $\theta$, $v_{i_{1}},...,v_{i_{5}}$ are consecutive under $\theta$, and $v_{i_{6}},v_{i_{7}}$ are consecutive under $\theta$. A \textit{right $\beta_{1}$-asteroid} (resp. \textit{right $\beta_{2}$-asteroid})  $\mathcal{A}^{\beta} := \lbrace v_{i_{1}},v_{i_{2}},v_{i_{3}},v_{i_{4}},v_{i_{5}},v_{i_{6}},v_{i_{7}} \rbrace$ \textit{of $T$ under} $\theta$ (where $i_{1}<...<i_{7}$) is an induced subtournament of $T$ with vertex set $\lbrace v_{i_{1}},v_{i_{2}},v_{i_{3}},v_{i_{4}},v_{i_{5}},v_{i_{6}},v_{i_{7}} \rbrace$, such that $\mathcal{A}^{\beta}$ is a right $\beta_{1}$-asteroid (resp. right $\beta_{2}$-asteroid) and has the forest ordering $(v_{i_{1}},v_{i_{2}},v_{i_{3}},v_{i_{4}},v_{i_{5}},v_{i_{6}},v_{i_{7}})$ under $\theta$, $v_{i_{1}},v_{i_{2}}$ are consecutive under $\theta$, and $v_{i_{3}},...,v_{i_{7}}$ are consecutive under $\theta$. A \textit{$\beta_{1}$-asteroid of $T$ under $\theta$} is a right or left $\beta_{1}$-asteroid of $T$ under $\theta$. A \textit{$\beta_{2}$-asteroid of $T$ under $\theta$} is a right or left $\beta_{2}$-asteroid of $T$ under $\theta$. A \textit{$\beta$-asteroid of $T$ under $\theta$} is a $\beta_{1}$-asteroid or $\beta_{2}$-asteroid of $T$ under $\theta$.\vspace{1mm} 

A tournament $T$ is an \textit{asterism} if there exists an ordering $ \theta $ of its vertices such that $V(T)$ is the disjoint union of $V(\mathcal{A}^{\beta}_{1}),...,V(\mathcal{A}^{\beta}_{l}),X$ where $\mathcal{A}^{\beta}_{1},...,\mathcal{A}^{\beta}_{l}$ are the $\beta$-asteroids of $T$ under $\theta$, and $T$$\mid$$X$ is a galaxy under $\overline{\theta}$ ($\overline{\theta}$ is the restriction of $\theta$ to $X$), and no vertex of a $\beta$-asteroid appears in the ordering $\theta$ between leaves of a star of $T$$\mid$$X$ under $\overline{\theta}$. We also say that $T$ \textit{is an \textit{asterism} under $\theta$} and $\theta$ is an \textit{asterism ordering of $T$}. If $T$$\mid$$X$ is a regular galaxy under $\overline{\theta}$, then $T$ is called a \textit{regular asterism}. \vspace{3mm}

In what follows we give examples of tournaments having $\beta$-asteroids (that is a $4$-path with a star) and centers of stars between leaves of another stars, that belongs to the family of so-called asterisms and are neither galaxies nor constellations (also not galaxies with spiders):\vspace{2.5mm}\\
-  The nine-vertex tournament $A_1$, such that $\{(8,5),(6,2),(5,2),(4,1),(9,3),(9,7)\}$ is the set of backward arcs of $A_1$ under the ordering $\gamma_1 =(1,...,9)$ of its vertices. Notice that $A_1$ is formed by merging the vertex disjoint right $\beta_{1}$-asteroid  $\{2,3,5,...,9\}$ and the right star  $\{1,4\}$ of $A_1$ under $\gamma_{1}$. In other words the connected components of $B(A_1,\gamma_{1})$ are the two right stars $\{3,7,9\}$ and $\{1,4\}$, and the $4$-path with vertex set $\{2,5,6,8\}$. Under $\gamma_{1}$, a new configuration of connected components appears (the $4$-path), and the center $4$ of the right star $\{1,4\}$ appears between the leaves of the right star $\{3,7,9\}$ (see Figure \ref{fig:asterismexamples}).   \vspace{1.5mm}\\
- $A_2$ is the $12$-vertex tournament, with $\{(11,4),(11,5),(5,2),(3,1),(12,1),(8,6),(10,6),(9,7)\}$ is the set of backward arcs of $A_2$ under the ordering $\gamma_{2}=(1,...,12)$ of its vertices. Under $\gamma_{2}$, $A_2$ is the disjoint union of the left $\beta_{2}$-asteroid  $\{1,...,5,11,12\}$ and the two left stars  $\{6,8,10\}$ and $\{7,9\}$. So the $4$-path with vertex set $\{2,4,5,11\}$ is one of the connected components of $B(A_2,\gamma_{2})$, and clearly $6$ and $7$ are centers of stars $\{6,8,10\}$ and $\{7,9\}$ respectively, appearing between the leaves $3$ and $12$ of the star $\{1,3,12\}$ (see Figure \ref{fig:asterismexamples}). \vspace{1.5mm}\\
- $A_3$ is the $16$-vertex tournament, with $\{(10,2),(11,2),(13,10),(16,7),(16,8),(8,5),(6,4),(15,4),(14,1),(14,$ $12),(9,3)\}$ is the set of backward arcs of $A_3$ under the ordering $\gamma_{3}=(1,...,16)$ of its vertices. Clearly the connected components of $B(A_3,\gamma_{3})$ are: The two four-paths $P_1$ and $P_2$ with vertex sets $\{2,10,11,13\}$ and $\{5,7,8,16\}$ respectively, the right star $\mathcal{S}_1:=\{1,12,14\}$, and the two left stars $\mathcal{S}_2:=\{3,9\}$ and $\mathcal{S}_3:=\{4,6,15\}$. One can clearly notice that the center of $\mathcal{S}_1$ is between the leaves of $\mathcal{S}_3$, and the centers of $\mathcal{S}_2$ and $\mathcal{S}_3$ are between the leaves of $\mathcal{S}_1$ (see Figure \ref{fig:asterismexamples}).\vspace{2mm}
     
For all $i\in[3]$, one can easily check that there exists no ordering of $V(A_i)$ such that the connected components of $A_i$ under this ordering are only stars.
 \begin{figure}[h]
	\centering
	\includegraphics[width=0.85\linewidth]{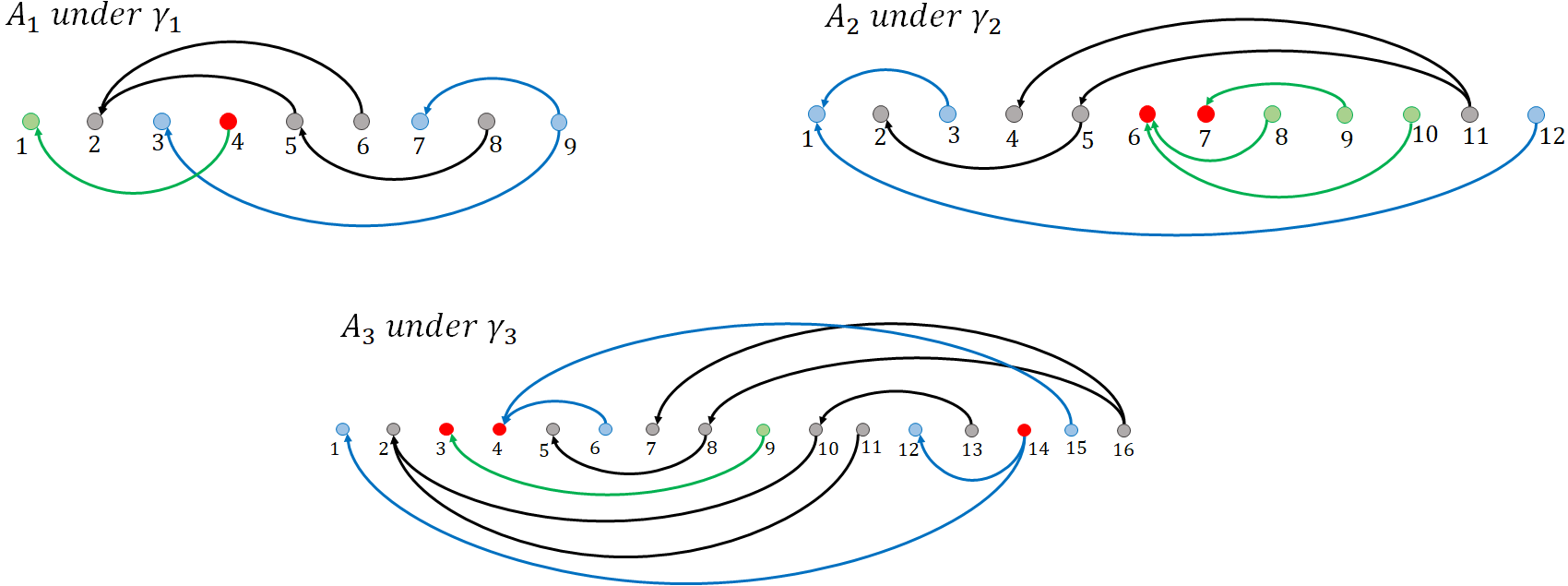}
	\caption{Asterisms drawn under their asterism ordering. The arcs corresponding to four-paths are colored in black, stars corresponding to $\beta$-asteroids are colored in blue, and other stars are colored in green. Centers of stars located between leaves of another stars are colored in red.  All the nondrawn arcs are forward.}
	\label{fig:asterismexamples}
\end{figure}

The main result in this section is:
\begin{theorem} \label{j} 
Every asterism satisfies EHC.
\end{theorem}

In what follows we introduce the tools that we heavily rely on in the proof of Theorem \ref{j}.
\begin{definition}
\normalfont Let $D$ be a tournament with vertices $v_{1},...,v_{7}$, and let $\theta_{1}=(v_{1},v_{2},v_{3},v_{4},v_{5},v_{6},v_{7})$ be an ordering of $V(D)$. Operation $1$ is the permutation of the vertices $v_{1},...,v_{7}$ that converts the ordering $\theta_{1}$ to the ordering $\theta_{2}=(v_{3},v_{1},v_{5},v_{2},v_{4},v_{6},v_{7})$ of $V(D)$, and operation $2$ is the permutation of the vertices $v_{1},...,v_{7}$ that converts the ordering $\theta_{1}$ to the ordering to the ordering $\theta_{3}=(v_{1},v_{2},v_{4},v_{6},v_{3},v_{7},v_{5})$ of $V(D)$.
\end{definition}

Let $H$ be a regular asterism under an ordering $\theta = (v_{1},...,v_{h})$ of its vertices with $\mid$$H$$\mid$ $= h$. Let $\mathcal{A}^{\beta}_{1},...,\mathcal{A}^{\beta}_{l}$ be the $\beta$-asteroids of $H$ under $\theta$. Let $X=V(H)\backslash \bigcup_{i=1}^{l}V(\mathcal{A}^{\beta}_{i}) $ and let $\overline{\theta} $ be the restriction of $\theta$ to $X$. Let $Q_{1},...,Q_{m}$ be the stars of $H$$\mid$$X$ under $\overline{\theta}$. In what follows we will assume without loss of generality that $m=l$.\\
To prove EHC for asterisms we studied the structure of the graph of backward arcs under several vertex orderings. Define $\Theta := \lbrace \theta^{'}$ an ordering of $V(H)$: $\theta^{'}$ is obtained from $\theta$ by performing operation $1$ to the vertex set of some left $\beta_{1}$-asteroids and some left $\beta_{2}$-asteroids of $H$ under $\theta$, and performing operation $2$ to the vertex set of some right $\beta_{1}$-asteroids and some right $\beta_{2}$-asteroids of $H$ under $\theta$$\rbrace$. Notice that $\mid$$\Theta$$\mid$ $= 2^{l}$. Clearly from Figure \ref{fig:generalizedasteroid}, one can notice the structure of the backward arcs when switching from one ordering to another in $\Theta$.  
\begin{remark}
Unlike galaxies and constellations, in asterisms, clearly the proof of Lemmas \ref{s} and \ref{r}  implies that the asterism ordering alone is not enough. In other words all the orderings in $\Theta$ are essential for confirming the conjecture for asterisms, as $H$ has different special configuration of backward arcs under each ordering in $\Theta$.
\end{remark}

 In what follows we explain how we construct the \textit{corresponding digraph of a given asterism}, that plays a central role in the proof of Theorem \ref{j} (This digraph together with the set of orderings $\Theta$ allows the breakthrough to prove the conjecture for asterisms).\vspace{2mm}\\
$\ast$ Let $  i\in [l]$, such that $\mathcal{A}^{\beta}_{i} = \lbrace v_{s_{i}},v_{s_{i}+1},v_{s_{i}+2},v_{s_{i}+3},v_{s_{i}+4},v_{q_{i}-1},v_{q_{i}} \rbrace$ is a left $\beta_{1}$$-$asteroid of $H$ under $\theta$. The \textit{mutant left $\beta_{1}$-asteroid} $\widehat{\mathcal{A}}^{\beta}_{i}$ (equivalently: the \textit{digraph corresponding to} $\mathcal{A}^{\beta}_{i}$ \textit{under} $\theta$) is the $13$-vertex digraph  that is obtained from $\mathcal{A}^{\beta}_{i}$ by deleting the arcs
$(v_{s_{i}+2},v_{q_{i}-1}),(v_{s_{i}+1},v_{q_{i}}),(v_{s_{i}+3},v_{s_{i}+4})$ and adding six extra vertices $x_{i}$ just before $v_{s_{i}}$, $g_{i}$ just before $x_{i}$, $m_{i}$ just before $g_{i}$, $r_{i}$ just after  $v_{s_{i}+3}$, $w_{i}$ just after $v_{s_{i}}$, and $y_{i}$ just before $v_{q_{i}-1}$, such that
 $m_{i}\rightarrow V(\widehat{\mathcal{A}}^{\beta}_{i})\backslash\lbrace r_{i}\rbrace ,g_{i}\rightarrow V(\widehat{\mathcal{A}}^{\beta}_{i})\backslash\lbrace m_{i},w_{i}\rbrace ,x_{i}\rightarrow V(\widehat{\mathcal{A}}^{\beta}_{i})\backslash\lbrace m_{i},g_{i},y_{i}\rbrace $, $v_{s_{i}}\rightarrow w_{i}\rightarrow V(\widehat{\mathcal{A}}^{\beta}_{i})\backslash\lbrace v_{s_{i}},m_{i},x_{i}\rbrace$, $\lbrace v_{s_{i}},v_{s_{i}+1},v_{s_{i}+2},v_{s_{i}+3}\rbrace\rightarrow r_{i}\rightarrow \lbrace m_{i},v_{s_{i}+4},y_{i},v_{q_{i}-1},v_{q_{i}}\rbrace $, and
$V(\widehat{\mathcal{A}}^{\beta}_{i})\backslash\lbrace x_{i},$ $v_{q_{i}-1},v_{q_{i}}\rbrace\rightarrow y_{i}\rightarrow \lbrace x_{i},v_{q_{i}-1},v_{q_{i}}\rbrace $. 
  We write $\widehat{\mathcal{A}}^{\beta}_{i} = \lbrace m_{i},g_{i},x_{i},v_{s_{i}},w_{i},v_{s_{i}+1},v_{s_{i}+2},v_{s_{i}+3},r_{i},v_{s_{i}+4},y_{i},v_{q_{i}-1},v_{q_{i}} \rbrace$ and we call $(m_{i},g_{i},x_{i},v_{s_{i}},w_{i},v_{s_{i}+1},v_{s_{i}+2},v_{s_{i}+3},r_{i},v_{s_{i}+4},y_{i},v_{q_{i}-1},v_{q_{i}})$ the \textit{forest ordering of $\widehat{\mathcal{A}}^{\beta}_{i}$} and $x_{i},w_{i},v_{s_{i}+1},v_{s_{i}+2},$ $v_{s_{i}+3},r_{i},v_{q_{i}-1}$ the \textit{leaves of} $\widehat{\mathcal{A}}^{\beta}_{i}$ (see Figure \ref{fig:betadigraphs}). \vspace{2mm}\\
$\ast$ Let $  i\in [l]$,  such that $\mathcal{A}^{\beta}_{i} = \lbrace v_{s_{i}},v_{s_{i}+1},v_{s_{i}+2},v_{s_{i}+3},v_{s_{i}+4},v_{q_{i}},v_{q_{i}+1} \rbrace$ is a left $\beta_{2}$-asteroid of $H$ under $\theta$. The \textit{mutant left $\beta_{2}$-asteroid} $\widehat{\mathcal{A}}^{\beta}_{i}$ (equivalently: the \textit{digraph corresponding to} $\mathcal{A}^{\beta}_{i}$ \textit{under} $\theta$) is the $13$-vertex digraph  that is obtained from $\mathcal{A}^{\beta}_{i}$ by deleting the arcs
$(v_{s_{i}+2},v_{q_{i}+1}),(v_{s_{i}+1},v_{q_{i}}),(v_{s_{i}+3},v_{s_{i}+4})$ and adding six extra vertices $x_{i}$ just before $v_{s_{i}}$, $g_{i}$ just before $x_{i}$, $m_{i}$ just before $g_{i}$, $r_{i}$ just after  $v_{s_{i}+3}$, $w_{i}$ just after $v_{s_{i}}$, and $y_{i}$ just after $v_{q_{i}+1}$, such that
 $m_{i}\leftarrow r_{i}, g_{i}\leftarrow w_{i},x_{i}\leftarrow y_{i}, m_{i}\rightarrow V(\widehat{\mathcal{A}}^{\beta}_{i})\backslash\lbrace r_{i}\rbrace ,g_{i}\rightarrow V(\widehat{\mathcal{A}}^{\beta}_{i})\backslash\lbrace m_{i},w_{i}\rbrace ,x_{i}\rightarrow V(\widehat{\mathcal{A}}^{\beta}_{i})\backslash\lbrace m_{i},g_{i},y_{i}\rbrace $, $v_{s_{i}}\rightarrow w_{i}\rightarrow V(\widehat{\mathcal{A}}^{\beta}_{i})\backslash\lbrace v_{s_{i}},m_{i},x_{i}\rbrace$, $\lbrace v_{s_{i}},v_{s_{i}+1},v_{s_{i}+2},v_{s_{i}+3}\rbrace\rightarrow r_{i}\rightarrow \lbrace v_{s_{i}+4},v_{q_{i}},v_{q_{i}+1},y_{i}\rbrace $, and
$V(\widehat{\mathcal{A}}^{\beta}_{i})\backslash\lbrace x_{i}\rbrace\rightarrow y_{i} $. 
  We write $\widehat{\mathcal{A}}^{\beta}_{i} = \lbrace m_{i},g_{i},x_{i},v_{s_{i}},w_{i},v_{s_{i}+1},v_{s_{i}+2},v_{s_{i}+3},r_{i},v_{s_{i}+4},v_{q_{i}},v_{q_{i}+1},y_{i} \rbrace$ and we call $(m_{i},g_{i},x_{i},v_{s_{i}},w_{i},v_{s_{i}+1},v_{s_{i}+2},v_{s_{i}+3},r_{i},v_{s_{i}+4},v_{q_{i}},v_{q_{i}+1},y_{i})$ the \textit{forest ordering of $\widehat{\mathcal{A}}^{\beta}_{i}$} and $x_{i},w_{i},v_{s_{i}+1},v_{s_{i}+2},$ $v_{s_{i}+3},r_{i},v_{q_{i}+1}$ the \textit{leaves of} $\widehat{\mathcal{A}}^{\beta}_{i}$ (see Figure \ref{fig:betadigraphs}). \vspace{2mm}\\ 
$\ast$ Let $  i\in [l]$,  such that $\mathcal{A}^{\beta}_{i} = \lbrace v_{q_{i}},v_{q_{i}+1},v_{s_{i}},v_{s_{i}+1},v_{s_{i}+2},v_{s_{i}+3},v_{s_{i}+4}\rbrace$ is a right $\beta_{1}$$-$asteroid of $H$ under $\theta$. The \textit{mutant right $\beta_{1}$-asteroid} $\widehat{\mathcal{A}}^{\beta}_{i}$ (equivalently: the \textit{digraph corresponding to} $\mathcal{A}^{\beta}_{i}$ \textit{under} $\theta$) is the $13$-vertex digraph  that is obtained from $\mathcal{A}^{\beta}_{i}$ by deleting the arcs
$(v_{q_{i}+1},v_{s_{i}+2}),(v_{q_{i}},v_{s_{i}+3}),(v_{s_{i}},v_{s_{i}+1})$ and adding six extra vertices $x_{i}$ just after $v_{s_{i}+4}$, $g_{i}$ just after $x_{i}$, $m_{i}$ just after $g_{i}$, $r_{i}$ just after  $v_{s_{i}}$, $w_{i}$ just after $v_{s_{i}+3}$, and $y_{i}$ just after $v_{q_{i}+1}$, such that
 $w_{i}\leftarrow g_{i}, V(\widehat{\mathcal{A}}^{\beta}_{i})\backslash\lbrace r_{i}\rbrace\rightarrow m_{i} , V(\widehat{\mathcal{A}}^{\beta}_{i})\backslash\lbrace m_{i},w_{i}\rbrace\rightarrow g_{i} , V(\widehat{\mathcal{A}}^{\beta}_{i})\backslash\lbrace m_{i},g_{i},y_{i}\rbrace \rightarrow x_{i}$, $V(\widehat{\mathcal{A}}^{\beta}_{i})\backslash\lbrace v_{s_{i}+4},x_{i},m_{i}\rbrace\rightarrow w_{i}\rightarrow v_{s_{i}+4} $, $\lbrace m_{i},v_{q_{i}},v_{q_{i}+1},y_{i},v_{s_{i}}\rbrace\rightarrow r_{i}\rightarrow \lbrace v_{s_{i}+1},v_{s_{i}+2},v_{s_{i}+3}, v_{s_{i}+4}\rbrace $, and
$\lbrace x_{i},v_{q_{i}},v_{q_{i}+1} \rbrace\rightarrow y_{i}\rightarrow V(\widehat{\mathcal{A}}^{\beta}_{i})\backslash\lbrace x_{i},v_{q_{i}},v_{q_{i}+1}\rbrace$. 
  We write $\widehat{\mathcal{A}}^{\beta}_{i} = \lbrace v_{q_{i}},v_{q_{i}+1},y_{i},v_{s_{i}},r_{i},v_{s_{i}+1},v_{s_{i}+2},v_{s_{i}+3},w_{i},v_{s_{i}+4}\\,x_{i},g_{i},m_{i} \rbrace$ and we call $(v_{q_{i}},v_{q_{i}+1},y_{i},v_{s_{i}},r_{i},v_{s_{i}+1},v_{s_{i}+2},v_{s_{i}+3},w_{i},v_{s_{i}+4},x_{i},g_{i},m_{i})$ the \textit{forest ordering of $\widehat{\mathcal{A}}^{\beta}_{i}$} and $v_{q_{i}+1},r_{i},v_{s_{i}+1},v_{s_{i}+2},v_{s_{i}+3},w_{i},x_{i}$ the \textit{leaves of} $\widehat{\mathcal{A}}^{\beta}_{i}$ (see Figure \ref{fig:betadigraphs}). \vspace{2mm}\\
$\ast$ Let $  i\in [l]$,  such that $\mathcal{A}^{\beta}_{i} = \lbrace v_{q_{i}-1},v_{q_{i}},v_{s_{i}},v_{s_{i}+1},v_{s_{i}+2},v_{s_{i}+3},v_{s_{i}+4}\rbrace$ is a right $\beta_{2}$$-$asteroid of $H$ under $\theta$. The \textit{mutant right $\beta_{2}$-asteroid} $\widehat{\mathcal{A}}^{\beta}_{i}$ (equivalently: the \textit{digraph corresponding to} $\mathcal{A}^{\beta}_{i}$ \textit{under} $\theta$) is the $13$-vertex digraph  that is obtained from $\mathcal{A}^{\beta}_{i}$ by deleting the arcs
$(v_{q_{i}-1},v_{s_{i}+2}),(v_{q_{i}},v_{s_{i}+3}),(v_{s_{i}},v_{s_{i}+1})$ and adding six extra vertices $x_{i}$ just after $v_{s_{i}+4}$, $g_{i}$ just after $x_{i}$, $m_{i}$ just after $g_{i}$, $r_{i}$ just after  $v_{s_{i}}$, $w_{i}$ just after $v_{s_{i}+3}$, and $y_{i}$ just before $v_{q_{i}-1}$, such that
 $r_{i}\leftarrow m_{i}, w_{i}\leftarrow g_{i},y_{i}\leftarrow x_{i}, V(\widehat{\mathcal{A}}^{\beta}_{i})\backslash\lbrace r_{i}\rbrace\rightarrow m_{i} , V(\widehat{\mathcal{A}}^{\beta}_{i})\backslash\lbrace m_{i},w_{i}\rbrace\rightarrow g_{i} , V(\widehat{\mathcal{A}}^{\beta}_{i})\backslash\lbrace m_{i},g_{i},y_{i}\rbrace \rightarrow x_{i}$, $V(\widehat{\mathcal{A}}^{\beta}_{i})\backslash\lbrace v_{s_{i}+4},x_{i},m_{i}\rbrace\rightarrow w_{i}\rightarrow v_{s_{i}+4} $, $\lbrace v_{q_{i}-1},v_{q_{i}},y_{i},v_{s_{i}}\rbrace\rightarrow r_{i}\rightarrow \lbrace v_{s_{i}+1},v_{s_{i}+2},v_{s_{i}+3}, v_{s_{i}+4}\rbrace $, and
$ y_{i}\rightarrow V(\widehat{\mathcal{A}}^{\beta}_{i})\backslash \lbrace x_{i}\rbrace $. 
  We write $\widehat{\mathcal{A}}^{\beta}_{i} = \lbrace y_{i},v_{q_{i}-1},v_{q_{i}},v_{s_{i}},r_{i},v_{s_{i}+1},v_{s_{i}+2},v_{s_{i}+3},w_{i},v_{s_{i}+4},x_{i},g_{i},m_{i} \rbrace$ and we call $(y_{i},v_{q_{i}-1},\\v_{q_{i}},v_{s_{i}},r_{i},v_{s_{i}+1},v_{s_{i}+2},v_{s_{i}+3},w_{i},v_{s_{i}+4},x_{i},g_{i},m_{i})$ the \textit{forest ordering of $\widehat{\mathcal{A}}^{\beta}_{i}$} and $v_{q_{i}-1},r_{i},v_{s_{i}+1},v_{s_{i}+2},v_{s_{i}+3},w_{i},x_{i}$ the \textit{leaves of} $\widehat{\mathcal{A}}^{\beta}_{i}$ (see Figure \ref{fig:betadigraphs}). 
   \vspace{3.5mm} 
   
   A \textit{mutant $\beta$-asteroid} is a mutant left or right $\beta_{1}$-asteroid or mutant left or right $\beta_{2}$-asteroid.\vspace{3mm}
   
Now we are ready to define the \textit{corresponding digraph of an asterism}:\vspace{2.5mm}\\Let $\widehat{\mathcal{A}}^{\beta}_{i}$ be the digraph corresponding to $\mathcal{A}^{\beta}_{i}$ for $i=1,...,l$. Let $\widehat{H}$ be the digraph obtained from $H$ by transforming $\mathcal{A}^{\beta}_{i}$ to $\widehat{\mathcal{A}}^{\beta}_{i}$ for $i=1,...,l$, and let $\widehat{\theta}$ be the obtained ordering of $V(\widehat{H})$. That is $\widehat{\theta}$ is obtained from $\theta$ by replacing the vertices of $\mathcal{A}^{\beta}_{i}$ by the vertices of $\widehat{\mathcal{A}}^{\beta}_{i}$ for $i=1,...,l$, such that $\widehat{\mathcal{A}}^{\beta}_{i}$ has its forest ordering under $\widehat{\theta}$, and:
\begin{itemize}
\item if $v_{s_{i}}$ is a vertex of a left $\beta_{1}$-asteroid, then $m_{i},g_{i},x_{i},v_{s_{i}},w_{i},v_{s_{i}+1},v_{s_{i}+2},v_{s_{i}+3},r_{i},v_{s_{i}+4}$ are consecutive  under $\widehat{\theta}$ and $y_{i},v_{q_{i}-1},v_{q_{i}}$ are consecutive  under $\widehat{\theta}$,
\item if $v_{s_{i}}$ is a vertex of a left $\beta_{2}$-asteroid, then $m_{i},g_{i},x_{i},v_{s_{i}},w_{i},v_{s_{i}+1},v_{s_{i}+2},v_{s_{i}+3},r_{i},v_{s_{i}+4}$ are consecutive  under $\widehat{\theta}$ and $v_{q_{i}},v_{q_{i}+1},y_{i}$ are consecutive  under $\widehat{\theta}$, 
\item if $v_{s_{i}}$ is a vertex of a right $\beta_{1}$-asteroid, then $v_{q_{i}},v_{q_{i}+1},y_{i}$ are consecutive  under $\widehat{\theta}$ and $v_{s_{i}},r_{i},v_{s_{i}+1},v_{s_{i}+2},$ $v_{s_{i}+3},w_{i},v_{s_{i}+4},x_{i},g_{i},m_{i}$ are consecutive  under $\widehat{\theta}$,
\item  if $v_{s_{i}}$ is a vertex of a right $\beta_{2}$-asteroid, then  $y_{i},v_{q_{i}-1},v_{q_{i}}$ are consecutive  under $\widehat{\theta}$ and $v_{s_{i}},r_{i},v_{s_{i}+1},v_{s_{i}+2},$ $v_{s_{i}+3},w_{i},v_{s_{i}+4},x_{i},g_{i},m_{i}$ are consecutive  under $\widehat{\theta}$. 
     \end{itemize} 
      We have $V(\widehat{H}) = V(H)\cup \bigcup_{i=1}^{l}X_i$ and
  $E(\widehat{H})= (E(H)\backslash \displaystyle{\bigcup_{i=1}^{l} E(\mathcal{A}^{\beta}_{i}))}\cup \displaystyle{\bigcup_{i=1}^{l}} E(\widehat{\mathcal{A}}^{\beta}_{i})\cup [[\displaystyle{\bigcup_{i=1}^{l}}$$(\displaystyle{\bigcup_{p\in X_{i}}}$$(\lbrace (x,p): x<_{\widehat{\theta}}p$ and $x\in V(\widehat{H}) \rbrace \cup $$\lbrace (p,x): p<_{\widehat{\theta}}x$ and $x\in V(\widehat{H}) \rbrace ))]\backslash \displaystyle{\bigcup_{i=1}^{l}} (\lbrace(x_{i},y_{i}): x_{i}<_{\widehat{\theta}}$ $ y_{i} \rbrace \cup \lbrace(y_{i},x_{i}): y_{i}<_{\widehat{\theta}}x_{i} \rbrace \cup \lbrace(m_{i},r_{i}): m_{i}<_{\widehat{\theta}}$ $ r_{i} \rbrace \cup \lbrace(r_{i},m_{i}): r_{i}<_{\widehat{\theta}}m_{i} \rbrace \cup \lbrace(g_{i},w_{i}): g_{i}<_{\widehat{\theta}}$ $ w_{i} \rbrace \cup \lbrace(w_{i},g_{i}): w_{i}<_{\widehat{\theta}}g_{i} \rbrace )]$, where $X_{i}=\lbrace m_{i},g_{i},x_{i},w_{i},r_{i},y_{i}\rbrace$. Note that $E(B(\widehat{H},\widehat{\theta}))$$ =$$ E(B(H,\theta))\cup\lbrace x_{i}y_{i}: i=1,...,l\rbrace\cup\lbrace m_{i}r_{i}: i=1,...,l\rbrace\cup\lbrace g_{i}w_{i}: i=1,...,l\rbrace$.\\ We say that $\widehat{H}$ is the \textit{digraph corresponding to} $H$ \textit{under} $\theta$, and $\widehat{\theta}$ is the \textit{ordering of $V(\widehat{H})$ corresponding to} $\theta$.
  \begin{figure}[h]
	\centering
	\includegraphics[width=0.9\linewidth]{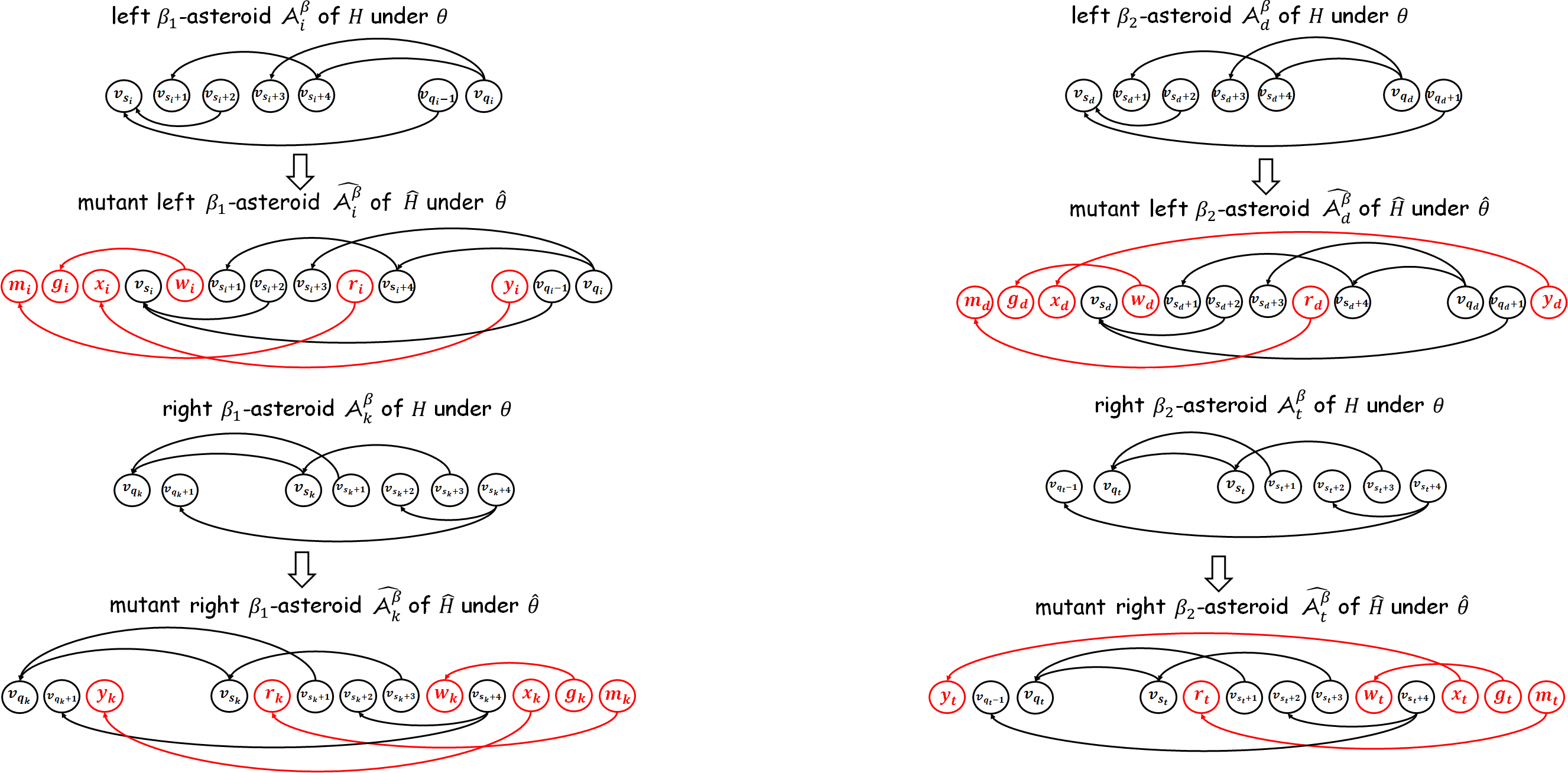}
	\caption{$\beta$$-$asteroids of $H$ under $\theta$: $\mathcal{A}_{i}^{\beta},\mathcal{A}_{k}^{\beta},\mathcal{A}_{d}^{\beta},\mathcal{A}_{t}^{\beta},$ and their corresponding mutant $\beta$-asteroids of $\widehat{H}$ under $\widehat{\theta}$: $\widehat{\mathcal{A}}_{i}^{\beta},\widehat{\mathcal{A}}_{k}^{\beta},\widehat{\mathcal{A}}_{d}^{\beta},$ and $\widehat{\mathcal{A}}_{t}^{\beta}$ respectively. The backward arcs are drawn. All the nondrawn arcs are forward except that the arcs $(v_{s_{i}+2},v_{q_{i}-1}),(v_{s_{i}+1},v_{q_{i}}),(v_{s_{i}+3},v_{s_{i}+4})\notin \widehat{\mathcal{A}}_{i}^{\beta}$, $(v_{q_{k}+1},v_{s_{k}+2}),(v_{s_{k}},v_{s_{k}+1}),(v_{q_{k}},v_{s_{k}+3})\notin \widehat{\mathcal{A}}_{k}^{\beta}$, $(v_{s_{d}+2},v_{q_{d}+1}),(v_{s_{d}+3},v_{s_{d}+4}),(v_{s_{d}+1},v_{q_{d}})\notin \widehat{\mathcal{A}}_{d}^{\beta}$, and $(v_{q_{t}-1},v_{s_{t}+2}),(v_{s_{t}},v_{s_{t}+1}),(v_{q_{t}},v_{s_{t}+3})\notin \widehat{\mathcal{A}}_{t}^{\beta}$.}
	\label{fig:betadigraphs}
\end{figure}
\begin{remark}
The corresponding digraph of a given asterism $H$ under $\theta$ is constructed using all backward arcs configurations of $H$ under $\theta^{'}$, for all $\theta^{'}\in \Theta$. So that whatever the outcomes and all possible cases that we have to study in the proof, we will be able to extract $H$ from an $\epsilon$-critical tournament $T$ using its corresponding digraph. Moreover $H$ will be viewed in $T$ under some ordering $\theta^{'}\in \Theta$. 
\end{remark}

Let $s$ be a $\lbrace 0,1 \rbrace$-vector. Denote $s_{c}$ the vector obtained from $s$ by replacing every subsequence of consecutive $1'$s by single $1$. 

Let $H$ be a regular asterism under an ordering $\theta$ of its vertices with $\mid$$H$$\mid$ $= h$.  Let $\mathcal{A}^{\beta}_{1},...,\mathcal{A}^{\beta}_{l}$ be the $\beta$-asteroids of $H$ under $\theta$, and $Q_{1},...,Q_{l}$ be the stars of $H$$\mid$$X$ under $\overline{\theta}$. Let $\widehat{\theta} = (u_{1},...,u_{h+6l})$ be the ordering of $V(\widehat{H})$ corresponding to $\theta$. 
For $i \in \lbrace 0,...,l \rbrace$ define $\widehat{H}^{i} = \widehat{H}$$\mid$$\bigcup_{j=1}^{i}(V(\widehat{\mathcal{A}}^{\beta}_{j})\cup V(Q_{j}))$ where $\widehat{H}^{l} = \widehat{H}$, and $\widehat{H}^{0}$ is the empty digraph. Let $s^{\widehat{H},\widehat{\theta}}$ be a $\lbrace 0,1 \rbrace$-vector such that $s^{\widehat{H},\widehat{\theta}}(i) = 1$ if and only if $u_{i}$ is a leaf of one of the stars of $\widehat{H}$ under $\widehat{\theta}$ or a leaf of one of $\widehat{\mathcal{A}}^{\beta}_{i}$ of $\widehat{H}$ under $\widehat{\theta}$ for $i=1,...,l$. Let $w = s_{c}^{\widehat{H},\widehat{\theta}}$ and let $i_{r}$ be such that $w_{i_{r}}=1$. Let $j$ be such that $s^{\widehat{H},\widehat{\theta}}_{j}=1$. We say that $s^{\widehat{H},\widehat{\theta}}_{j}$ \textit{corresponds to} $w_{i_{r}}$ if $s^{\widehat{H},\widehat{\theta}}_{j}$ belongs to the subsequence of consecutive $1'$s that is replaced by the entry $w_{i_{r}}$. For $k \in \lbrace 1,...,l \rbrace$, let $\widehat{\theta}_{k} = (v_{k_{1}},...,v_{k_{t_{k}}})$ with $t_{k}=$ $\mid$$ \widehat{H}^{k}$$ \mid$, be the restriction of $\widehat{\theta}$ to $V(\widehat{H}^{k})$. Let $s^{\widehat{H},\widehat{\theta}}_{\widehat{H}^{k}}$ be the restriction of $s^{\widehat{H},\widehat{\theta}}$ to the $0's$ and $1's$ corresponding to $V(\widehat{H}^{k})$ (notice that $s^{\widehat{H},\widehat{\theta}}_{\widehat{H}^{k}}= s^{\widehat{H}^{k},\widehat{\theta}_{k}}$) and let $^{c}s^{\widehat{H},\widehat{\theta}}_{\widehat{H}^{k}}$ be the vector obtained from $s^{\widehat{H},\widehat{\theta}}_{\widehat{H}^{k}}$ by replacing every subsequence of consecutive $1's$ corresponding to the same entry of $s^{\widehat{H},\widehat{\theta}}_{c}$ by single $1$. We say that a smooth $(c,\lambda ,w)$-structure of a tournament $T$ \textit{corresponds}  \textit{to $\widehat{H}^{k}$ under $(\widehat{H},\widehat{\theta})$} if $w =$ $ ^{c}s^{\widehat{H},\widehat{\theta}}_{\widehat{H}^{k}}$ (in case $k=l$, we simply say: a smooth $(c,\lambda ,w)$-structure of a tournament $T$ \textit{corresponds}  \textit{to $\widehat{H}$ under $\widehat{\theta}$}). Notice that $s^{\widehat{H},\widehat{\theta}}_{\widehat{H}^{l}}=s^{\widehat{H},\widehat{\theta}}$ and $^{c}s^{\widehat{H},\widehat{\theta}}_{\widehat{H}^{l}}=s^{\widehat{H},\widehat{\theta}}_{c}$.\\
Let $\nu =$ $^{c}s^{\widehat{H},\widehat{\theta}}_{\widehat{H}^{k}}$. Let $\delta^{\nu}:$ $\lbrace j: \nu_{j} = 1 \rbrace \rightarrow \mathbb{N}$ be a function that assigns to every nonzero entry of $\nu$ the number of consecutive $1'$s of $s^{\widehat{H},\widehat{\theta}}_{\widehat{H}^{k}}$ replaced by that entry of $\nu$.
Fix $k \in \lbrace 0,...,l \rbrace$. Let $(S_{1},...,S_{\mid w \mid})$ be a smooth $(c,\lambda ,w)$-structure corresponding to $\widehat{H}^{k}$ under $(\widehat{H},\widehat{\theta})$.
 \\We borrow the definition of \textit{well-contained} from \cite{polll}. Let $i_{r}$ be such that $w(i_{r}) = 1$. Assume that $S_{i_{r}} = \lbrace s^{1}_{i_{r}},...,s_{i_{r}}^{\mid S_{i_{r}} \mid} \rbrace$ and $(s^{1}_{i_{r}},...,s_{i_{r}}^{\mid S_{i_{r}} \mid})$ is a transitive ordering. Write $m(i_{r}) = \lfloor\frac{\mid S_{i_{r}} \mid}{\delta^{w}(i_{r})}\rfloor$. Denote $S^{j}_{i_{r}} = \lbrace s^{(j-1)m(i_{r})+1}_{i_{r}},$ $...,s_{i_{r}}^{jm(i_{r})} \rbrace$ for $j \in \lbrace 1,...,\delta^{w}(i_{r}) \rbrace$. For every $v \in S^{j}_{i_{r}}$ denote $\xi(v) = (\mid$$\lbrace k < i_{r}: w(k) = 0 \rbrace$$\mid$ $+$ $\displaystyle{\sum_{k < i_{r}: w(k) = 1}\delta^{w}(k) })$ $+$ $j$. For every $v \in S_{i_{r}}$ such that $w(i_{r}) = 0$ denote $\xi(v) = (\mid$$\lbrace k < i_{r}: w(k) = 0 \rbrace$$\mid$ $+$ $\displaystyle{\sum_{k < i_{r}: w(k) = 1}\delta^{w}(k) })$ $+$ $1$. We say that $\widehat{H}^{k}$ is \textit{well-contained in} $(S_{1},...,S_{\mid w \mid})$ that corresponds to $\widehat{H}^{k}$ under $(\widehat{H},\widehat{\theta})$ if there is an injective homomorphism $f$ of $\widehat{H}^{k}$ into $T$$\mid$$\bigcup_{i = 1}^{\mid w \mid}S_{i}$ such that $\xi(f(u_{k_j})) = j$ for every $j \in \lbrace 1,...,t_{k} \rbrace$, where $\widehat{\theta}_{k} = (u_{k_{1}},...,u_{k_{t_{k}}})$. 
 \begin{remark}
$^{c}s^{\widehat{H},\widehat{\theta}}_{\widehat{H}^{k}}$ and  $s_{c}^{\widehat{H}^{k},\widehat{\theta}_{k}}$ are not necessarily the same vector. It's the case if there exists a sequence  of consecutive $1$'s in $s^{\widehat{H}^{k},\widehat{\theta}_{k}}$ that are separated by a $0$ in $s^{\widehat{H},\widehat{\theta}}$. 
 \end{remark} 
 \subsection{An overview of the proof}
 
 \hspace{3.5mm} The proof of Theorem \ref{j} is very technical and bit hard. To this end we will give the reader an intuition how the proof works, what are the techniques that allows the breakthrough in asterisms, and what are the most important steps of the proof.
 
Let $H$ be an asterism under $\theta$. To prove that $H$ satisfies EHC, it is enough to find a copy of $H$ in an $\epsilon$-critical tournament $T$ for some $\epsilon >0$.\\  
 Only at the starting point we will proceed as in galaxies and constellations by assuming that $EHC$ is not true for $H$, and so Theoem \ref{i} implies that we can assume that there exists a smooth $(c,\lambda ,w)$-structure (in this step we did not need to assume any structural conditions on the forbidden tournament $H$, so this step is as in galaxies and constellations). What is different is that instead of assuming the existence of a $(c,\lambda ,w)$-structure corresponding to $H$ under $\theta$, we will assume the existence of a smooth $(c,\lambda ,w)$-structure $\chi$ corresponding to $\widehat{H}$ under $\widehat{\theta}$.  Clearly  Lemmas \ref{s} and \ref{r} does not imply that we can construct a copy of a $\beta$-asteroid. So we can't proceed here as in galaxies and constellations by constructing the  $\beta$-asteroids one by one following the ordering $\theta$ to get a copy of $H$. This implies that following the ordering $\theta$ alone will not work. Here is the point where we start thinking about another orderings for $H$ that gives us different configuration of the backward arcs than that obtained under the ordering $\theta$. As an outcome we form the set of orderings denoted by $\Theta$, that contains orderings of $H$, such that every connected component of the graph of backward arcs of $H$ under $\theta^{'}\in \Theta$ is a triangle or a $4$-path or a star. So in order to be able to extract from $T$ a copy of $H$ under $\theta^{'}$, for some $\theta^{'}\in \Theta$, we construct a copy of a mutant $\beta$-asteroid instead of a $\beta$-asteroid (since we don't necessarily get a $\beta$-asteroid as we said). The importance of constructing copies of mutant $\beta$-asteroids, denoted by $Y_i$, is that whatever the orientation of the arcs in $E(T$$\mid$$\bigcup_{i}Y_i)\backslash E(\bigcup_{i}Y_i)$ is, we can extract from $E(T$$\mid$$\bigcup_{i}Y_i)$ either a $4$-path or a triangle. So we construct in $T$ the corresponding digraph $\widehat{H}$ of $H$ under $\theta$, which is obtained from $H$ by replacing $\beta$-asteroids by mutant $\beta$-asteroids. Here we construct the mutant $\beta$-asteroids one by one after updating the sets in $\chi$ in a way that we can merge all the constructed copies together. Similarly we construct the rest of the stars and then merge them with the constructed copies of the mutant $\beta$-asteroids to get a copy of $\widehat{H}$ under $\widehat{\theta}$. It's the time to extract a copy of $H$ from the tournament in $T$ induced by the vertex set of the copy of $\widehat{H}$. Let's now order the vertices of $\widehat{H}$ according to their appearance in $\chi$. Denote this ordering by $\alpha$. Now for all $i$, we can extract either a triangle or a $\beta$-asteroid. If the former holds, we extract the two stars that forms together with the triangle the cyclic ordering of the $\beta$-asteroid. If the latter holds we extract the $\beta$-asteroid (here the $\beta$-asteroid appears in its forest ordering). Finally we delete from $\alpha$ all the vertices of the nonextracted substructures. The outcome is an ordering $\theta^{'}\in \Theta$. We are done.  
\subsection{Proof}
\hspace{3.5mm} In the following lemma we used Lemmas \ref{s} and \ref{r} to prove the existence of a mutant $\beta$-asteroid in an appropriate $\epsilon$-critical tournament. 
\begin{lemma} \label{k} 
Let $\widehat{\mathcal{A}}^{\beta}$ be a mutant $\beta$-asteroid and let $\gamma = (a_{1},...,a_{13})$ be its forest ordering. Let $c>0$, $0<\lambda <\frac{1}{70}$, $0<\epsilon< min \lbrace log_{\frac{c}{4}}(\frac{1}{2}), log_{\frac{c}{8}}(1-\frac{c}{10})\rbrace$ be constants, and $w$ a $\lbrace 0,1 \rbrace$-vector with $\mid$$w$$\mid$ $=9$. Let $(S_{1},...,S_{\mid w \mid})$ be a smooth $(c,\lambda ,w)$-structure of an $\epsilon$-critical tournament $T$ corresponding to $\widehat{\mathcal{A}}^{\beta}$ under $\gamma$, with $\mid$$T$$\mid$ $= n$. Then $\widehat{\mathcal{A}}^{\beta}$ is well-contained in $(S_{1},...,S_{\mid w \mid})$.  \end{lemma}
\begin{proof}
Assume that $\widehat{\mathcal{A}}^{\beta}$ is a mutant left $\beta_{1}$-asteroid and let $\gamma = (a_{1},...,a_{13})$ be its forest ordering (otherwise, the argument is similar, and we omit it). In this case $w=(0,0,1,0,1,0,0,1,0)$. By Lemma \ref{s}, since $T$ is $\epsilon$-critical and $\epsilon < min\lbrace log_{\frac{c}{4}}(\frac{1}{2}), log_{\frac{c}{2}}(1-\frac{c}{5})\rbrace$, there exist vertices $x^{2}_{5}\in S^{2}_{5}$, $x^{4}_{5}\in S^{4}_{5}$, $x_{6}\in S_{6}$, and $x_{9}\in S_{9}$, such that $\lbrace x_{5}^{4},x_{6}\rbrace \leftarrow x_{9}$ and $x_{5}^{2}\leftarrow x_{6}$. Let $\widehat{S_{5}^{3}} = \lbrace x_{5}^{3}\in S_{5}^{3}; x_{5}^{3} \rightarrow \lbrace x_{6},x_{9}\rbrace \rbrace$, $S_{8}^{*} = \lbrace x_{8}\in S_{8}; \lbrace x_{5}^{2},x_{5}^{4},x_{6}\rbrace\rightarrow x_{8} \rightarrow x_{9} \rbrace$, and $S_{4}^{*} = \lbrace x_{4}\in S_{4}; x_{4} \rightarrow \lbrace x_{5}^{2},x_{5}^{4},x_{6},x_{9}\rbrace \rbrace$. Then by Lemma \ref{g}, $\mid$$S_{4}^{*}$$\mid$ $\geq (1-4\lambda)cn\geq \frac{c}{2}n$ since $\lambda \leq \frac{1}{8}$, $\mid$$\widehat{S_{5}^{3}}$$\mid$ $\geq \frac{1-10\lambda}{5}ctr(T)\geq \frac{c}{10}tr(T)$ since $\lambda \leq \frac{1}{20}$, and $\mid$$S_{8}^{*}$$\mid$ $\geq (1-4\lambda)ctr(T)\geq \frac{c}{2}tr(T)$ since $\lambda \leq \frac{1}{8}$. Since $ \epsilon < log_{\frac{c}{8}}(1-\frac{c}{10}) $, then by Lemma \ref{r} there exist vertices $x_{4}\in S^{*}_{4} $, $x_{5}^{3}\in \widehat{S_{5}^{3}}$, and $x_{8}\in S^{*}_{8} $, such that $x_{4}\leftarrow \lbrace x_{5}^{3},x_{8}\rbrace$. Let $S_{3}^{*} = \lbrace x_{3}\in S_{3}; x_{3} \rightarrow \lbrace x_{4},x_{5}^{2},x_{5}^{3},x_{5}^{4},x_{6},x_{8},x_{9}\rbrace \rbrace$ and $S_{7}^{*} = \lbrace x_{7}\in S_{7}; \lbrace x_{4},x_{5}^{2},x_{5}^{3},x_{5}^{4},x_{6}\rbrace\rightarrow x_{7} \rightarrow \lbrace x_{8},x_{9}\rbrace \rbrace$. Then by Lemma \ref{g}, $\mid$$S_{3}^{*}$$\mid$ $\geq (1-7\lambda)ctr(T)\geq \frac{c}{2}tr(T)$ since $\lambda \leq \frac{1}{14}$, and $\mid$$S_{7}^{*}$$\mid$ $\geq (1-7\lambda)cn\geq \frac{c}{2}n$ since $\lambda \leq \frac{1}{14}$. 
 Since $\epsilon < log_{\frac{c}{2}}(1-\frac{c}{2})$, then Lemma \ref{f} implies that there exist vertices $x_{3}\in S_{3}^{*}$ and $x_{7}\in S_{7}^{*}$ such that $x_{3} \leftarrow x_{7}$. Let $S_{2}^{*} = \lbrace x_{2}\in S_{2}; x_{2} \rightarrow \lbrace x_{3},x_{4},x_{5}^{2},x_{5}^{3},x_{5}^{4},x_{6},x_{7},x_{8},x_{9}\rbrace \rbrace$ and $\widehat{S_{5}^{1}} = \lbrace x_{5}^{1}\in S_{5}^{1}; \lbrace x_{3},x_{4}\rbrace\rightarrow x_{5}^{1}\rightarrow \lbrace x_{6}, x_{7}, x_{8},x_{9}\rbrace \rbrace$. Then by Lemma \ref{g}, $\mid$$S_{2}^{*}$$\mid$ $\geq (1-9\lambda)cn\geq \frac{c}{2}n$ since $\lambda \leq \frac{1}{18}$ and $\mid$$\widehat{S_{5}^{1}}$$\mid$ $\geq \frac{1-30\lambda}{5}ctr(T)\geq \frac{c}{10}tr(T)$ since $\lambda \leq \frac{1}{60}$.
 Since $\epsilon < log_{\frac{c}{2}}(1-\frac{c}{10})$, then Lemma \ref{f} implies that there exist vertices $x_{2}\in S_{2}^{*}$ and $x_{5}^{1}\in \widehat{S_{5}^{1}}$ such that $x_{2} \leftarrow x_{5}^{1}$. Let $S_{1}^{*} = \lbrace x_{1}\in S_{1}; x_{1} \rightarrow \lbrace x_{2},x_{3},x_{4},x_{5}^{1},x_{5}^{2},x_{5}^{3},x_{5}^{4},x_{6},x_{7},x_{8},x_{9}\rbrace \rbrace$ and $\widehat{S_{5}^{5}} = \lbrace x_{5}^{5}\in S_{5}^{5}; \lbrace x_{2},x_{3},x_{4}\rbrace\rightarrow x_{5}^{5}\rightarrow \lbrace x_{6}, x_{7}, x_{8},x_{9}\rbrace \rbrace$. Then by Lemma \ref{g}, $\mid$$S_{1}^{*}$$\mid$ $\geq (1-11\lambda)cn\geq \frac{c}{2}n$ since $\lambda \leq \frac{1}{22}$ and $\mid$$\widehat{S_{5}^{5}}$$\mid$ $\geq \frac{1-35\lambda}{5}ctr(T)\geq \frac{c}{10}tr(T)$ since $\lambda \leq \frac{1}{70}$.
 Since $\epsilon < log_{\frac{c}{2}}(1-\frac{c}{10})$, then Lemma \ref{f} implies that there exist vertices $x_{1}\in S_{1}^{*}$ and $x_{5}^{5}\in \widehat{S_{5}^{5}}$ such that $x_{1} \leftarrow x_{5}^{5}$.  So, $ (x_{1},x_{2},x_{3},x_{4},x^{1}_{5},x^{2}_{5},x^{3}_{5},x^{4}_{5},x_{5}^{5},x_{6},x_{7},x_{8},x_{9})$ is the forest ordering of the mutant left $\beta_{1}$-asteroid. $\hfill{\square}$ \vspace{4mm}
\end{proof}

Constructing the corresponding digraph of a given asterism will be very handy in our latter analysis. To this end, for any asterism, we use the previous lemma to construct its corresponding digraph in an appropriate smooth $(c,\lambda ,w)$-structure.
\begin{lemma}\label{d} 
Let $H$ be a regular asterism under an ordering $\theta$ of its vertices with $\mid$$H$$\mid$ $= h$. Let $\mathcal{A}^{\beta}_{1},...,\mathcal{A}^{\beta}_{l}$ be the $\beta$-asteroids of $H$ under $\theta$, and let $Q_{1},...,Q_{l}$ be the stars of $H$$\mid$$(V(H)\backslash \bigcup_{i=1}^{l}V(\mathcal{A}^{\beta}_{i}))$ under $\overline{\theta}$. Let $0 < \lambda < \frac{1}{(4h)^{h+4}}$, $c > 0$ be constants, and $w$ be a $\lbrace 0,1 \rbrace$-vector. Fix $k \in \lbrace 0,...,l \rbrace$ and let $\widehat{\lambda} = (2h)^{l-k}\lambda$ and $\widehat{c} = \frac{c}{(2h)^{l-k}}$. There exist $ \epsilon_{k} > 0$ such that $\forall 0 < \epsilon < \epsilon_{k}$, for every $\epsilon$$-$critical tournament $T$ with $\mid$$T$$\mid$ $= n$ containing $\chi = (S_{1},...,S_{\mid w \mid})$ as a smooth $(\widehat{c},\widehat{\lambda},w)$-structure corresponding to $\widehat{H^{k}}$ under $(\widehat{H},\widehat{\theta})$, we have $\widehat{H^{k}}$ is well-contained in $\chi$.  
\end{lemma}
\begin{proof}
The proof is by induction on $k$. For $k=0$ the statement is obvious since $\widehat{H^{0}}$ is the empty digraph. Suppose that $\chi = (S_{1},...,S_{\mid w \mid})$ is a smooth $(\widehat{c},\widehat{\lambda},w)$-structure in $T$ corresponding to $\widehat{H}^{k}$ under $(\widehat{H},\widehat{\theta})$, with $\widehat{\theta}=(h_{1},...,h_{h+6l})$. Let $\widehat{\theta}_{k} = (h_{k_{1}},...,h_{k_{s}})$ be the restriction of $\widehat{\theta}$ to $V(\widehat{H}^{k})$, where $s=$ $\mid$$ \widehat{H}^{k}$$ \mid$. Let $\widehat{\mathcal{A}^{\beta}_{k}} = \lbrace h_{k_{q_{1}}},...,h_{k_{q_{13}}} \rbrace$. Let $h_{k_{p_{0}}}$ be the center of $Q_{k}$ and $h_{k_{p_{1}}},...,h_{k_{p_{q}}}$ be its leafs for some integer $q>0$. Let $D_{i} = \lbrace v \in \bigcup_{j=1}^{\mid w \mid}S_{j};$ $ \xi(v) = q_{i} \rbrace$ for $i=1,...,13$. Assume without loss of generality that $\mathcal{A}^{\beta}_{k}$ is a left $\beta_{1}$$-$asteroid. Then there exist $ 1 \leq y \leq \mid$$w$$\mid$, and there exist $ 1 \leq b \leq \mid$$w$$\mid$, such that $1 \leq y-3 < y+2<b<b+2\leq $ $\mid$$w$$\mid$ and $D_{1} = S_{y-3}$, $D_{2} = S_{y-2}$, $D_{3} = S_{y-1}$, $D_{4} = S_{y}$, $D_{5} = S^{1}_{y+1}$, $D_{6} = S^{2}_{y+1}$, $D_{7} = S^{3}_{y+1}$, $D_{8} = S^{4}_{y+1}$, $D_{9} = S^{5}_{y+1}$, $D_{10} = S_{y+2}$, $D_{11} = S_{b}$, $D_{12} = S_{b+1}$, $D_{13} = S_{b+2}$ with $w(y-3)=w(y-2)=w(y)=w(y+2)=w(b) =w(b+2)= 0$ and $w(y-1)=w(y+1) =w(b+1)= 1$. Then $\tilde{\chi}=(S_{y-3},...,S_{y+2},S_{b},S_{b+1},S_{b+2})$ is a  smooth $(\widehat{c},\widehat{\lambda},\tilde{w})$-structure in $T$ corresponding to $\widehat{\mathcal{A}^{\beta}_{k}}$ under $(h_{q_{1}},...,h_{q_{13}})$, with $\tilde{w}=(0,0,1,0,1,0,0,1,0)$. Since $\widehat{\lambda} <\frac{1}{70}$, and since we can assume that $\epsilon < min \lbrace log_{\frac{\widehat{c}}{4}}(\frac{1}{2}), log_{\frac{\widehat{c}}{8}}(1-\frac{\widehat{c}}{10})\rbrace$, then by Lemma \ref{k}, there exist vertices $x_{1},...,x_{13}$ such that $x_{i} \in D_{i}$ for $i=1,...,13$ and $T$$\mid$$\lbrace x_{1},...,x_{13}\rbrace$ contains a copy of $\widehat{H}^{k}$$\mid$$V(\widehat{\mathcal{A}^{\beta}_{k}})$ where $(x_{1},...,x_{13})$ is its forest ordering (that is $\widehat{\mathcal{A}^{\beta}_{k}}$ is well-contained in $\tilde{\chi}$).   
For all $ 0 \leq i \leq q$, let $R_{i} = \lbrace v \in \bigcup_{j=1}^{\mid w \mid}S_{j};$ $ \xi(v) = p_{i} \rbrace$ and let $R_{i}^{*} = \bigcap_{x\in V(W)}R_{i,x}$.
Then $\exists m \in \lbrace 1,...,\mid$$w$$\mid \rbrace \backslash \lbrace y-3,y-2,y,y+2,b,b+2 \rbrace$, with $w(m)=0$ and $\exists f \in \lbrace 1,...,\mid$$w$$\mid \rbrace \backslash \lbrace y-1,y+1,b+1 \rbrace$ with $w(f)=1$, such that $R_{0} = S_{m}$ and $\forall 1 \leq i \leq q$, $R_{i} \subseteq S_{f}$. Then by Lemma \ref{g}, $\mid$$R_{0}^{*}$$\mid$ $\geq (1-13\widehat{\lambda})\mid$$R_{0}$$\mid$ $\geq \frac{\mid R_{0}\mid}{2}$ $\geq\frac{\widehat{c}}{2}n$ since $\widehat{\lambda} \leq \frac{1}{26}$, and $\mid$$ R_{i}^{*}$$\mid$ $ \geq \frac{1-13h\widehat{\lambda}}{h}\mid$$ S_{f}$$ \mid$ $\geq \frac{\widehat{c}}{2h}tr(T)$ since $\widehat{\lambda} \leq \frac{1}{26h}$. Since we can assume that $\epsilon < log_{\frac{\widehat{c}}{4h}}(1-\frac{\widehat{c}}{2h})$, then by Lemma \ref{r} there exists vertices $r_{0},r_{1},...,r_{q}$ such that $r_{i} \in R_{i}^{*}$ for $i=0,1,...,q$ and \\
$\ast$ $r_{1},...,r_{q}$ are all adjacent from $r_{0}$ if $m>f$.\\
$\ast$ $r_{1},...,r_{q}$ are all adjacent to $r_{0}$ if $m<f$.\\
So $T$$\mid$$\lbrace x_{1},...,x_{13},r_{0},r_{1},...,r_{q} \rbrace$ contains a copy of $\widehat{H^{k}}$$\mid$$(V(\widehat{\mathcal{A}^{\beta}_{k}})\cup V(Q_{k}))$. Denote this copy by $Y$.\\    
For all $ i \in \lbrace 1,...,\mid$$w$$\mid \rbrace \backslash \lbrace y-3,...,y+2,b,b+1,b+2,m,f \rbrace$, let $S_{i}^{*} = \displaystyle{\bigcap_{x\in V(Y)}S_{i,x}}$.  Then by Lemma \ref{g}, $\mid$$S_{i}^{*}$$\mid$ $\geq (1-\mid$$Y$$\mid\widehat{\lambda})\mid$$S_{i}$$\mid$ $\geq (1-(h+6)\widehat{\lambda})\mid$$S_{i}$$\mid$ $\geq \frac{\mid S_{i}\mid}{2h}$ since $\widehat{\lambda} \leq \frac{2h-1}{2h(h+6)}$.
Write $\mathcal{H} = \lbrace 1,...,s \rbrace \backslash \lbrace q_{1},...,q_{13},p_{0},...,p_{q} \rbrace$. If $\lbrace v\in S_{f}: \xi(v) \in \mathcal{H} \rbrace \neq \phi$, then define $J_{f} = \lbrace \eta \in \mathcal{H}: \exists v \in S_{f}$ and $\xi(v)= \eta \rbrace$. Now for all $ \eta \in J_{f}$, let $S_{f}^{*\eta}= \lbrace v \in S_{f}: \xi(v)=\eta$ and $v \in \displaystyle{\bigcap_{x\in V(Y)\backslash\lbrace r_{1},...,r_{q} \rbrace}S_{f,x}} \rbrace$. Then by Lemma \ref{g}, for all $ \eta \in J_{f}$, we have $\mid$$S_{f}^{*\eta}$$\mid$ $ \geq \frac{1-14h\widehat{\lambda}}{h}\mid $$S_{f}$$\mid $ $\geq \frac{\mid S_{f}\mid}{2h}$ since $\widehat{\lambda} \leq \frac{1}{28h}$. Now for all $ \eta \in J_{f}$, select arbitrary $\lceil \frac{\mid S_{f}\mid}{2h}\rceil$ vertices of $S_{f}^{*\eta}$, and denote the union of these $\mid$$J_{f}$$\mid$ sets by $S_{f}^{*}$.   
So we have defined $t$ sets $S_{1}^{*},...,S^{*}_{t}$, where $t = \mid$$w$$\mid$ $-10$ if $S_{f}^{*}$ is defined, and $t = \mid$$w$$\mid$ $-11$ if $S_{f}^{*}$ is not defined. We have $\mid$$S_{i}^{*}$$\mid$ $\geq \frac{\widehat{c}}{2h}tr(T)$ for every defined $S_{i}^{*}$ with $w(i) = 1$, and $\mid$$S_{i}^{*}$$\mid$ $\geq \frac{\widehat{c}}{2h}n$ for every defined $S_{i}^{*}$ with $w(i) = 0$. Now Lemma \ref{b} implies that $\chi^{*}=(S_{1}^{*},...,S^{*}_{t})$ form a smooth $(\frac{\widehat{c}}{2h},2h\widehat{\lambda},w^{*})$-structure of $T$ corresponding to $\widehat{H}^{k-1}$ under  $(\widehat{H},\widehat{\theta})$, where $\frac{\widehat{c}}{2h}= \frac{c}{(2h)^{l-(k-1)}}, 2h\widehat{\lambda}=(2h)^{l-(k-1)}\lambda$, and $w^{*}$ is an appropriate $\lbrace 0,1 \rbrace$-vector.
Now take $\epsilon_{k} < min \lbrace \epsilon_{k-1}, log_{\frac{\widehat{c}}{4}}(\frac{1}{2}), log_{\frac{\widehat{c}}{8}}(1-\frac{\widehat{c}}{16}), log_{\frac{\widehat{c}}{4h}}(1-\frac{\widehat{c}}{2h}) \rbrace$. So by induction hypothesis $\widehat{H}^{k-1}$ is well-contained in $\chi^{*}$. Now by merging the well-contained copy of $\widehat{H}^{k-1}$ and $Y$ we get a well-contained copy of $\widehat{H}^{k}$. $\hfill{\square}$ 
\end{proof}\vspace{4mm}

In the following corollary we introduce a rule that uses the corresponding digraph of an asterism $H$, to find $H$ as an induced copy in $T$:

\begin{corollary} \label{n} 
Let $H$ be a regular asterism under an ordering $\theta$ of its vertices. Let $\mathcal{A}^{\beta}_{1},...,\mathcal{A}^{\beta}_{l}$ be the $\beta$-asteroids of $H$ under $\theta$, and let $Q_{1},...,Q_{l}$ be the stars of $H$$\mid$$(V(H)\backslash \bigcup_{i=1}^{l}V(\mathcal{A}^{\beta}_{i}))$ under $\overline{\theta}$. Let $ \lambda > 0$ ($\lambda$ is small enough), $c > 0$ be constants, and let $w$ be a $\lbrace 0,1 \rbrace$$-$vector. Suppose that  $\chi = (S_{1},...,S_{\mid w \mid})$ is a smooth $(c,\lambda ,w)$-structure of an $\epsilon$$-$critical tournament $T$ ($\epsilon$ is small enough) corresponding to $\widehat{H}$ under $\widehat{\theta}$. Then $T$ contains $H$.
\end{corollary}  
\begin{proof}
$\widehat{H}=\widehat{H}^{l}$ is well-contained in $\chi$ by the previous lemma when taking $k=l$. For all $ i \in [l]$, let $\widetilde{\mathcal{A}}_{i}^{\beta}= \lbrace c_{i},d_{i},b_{i},g_{i},n_{i},r_{i},p_{i},q_{i},o_{i},s_{i},f_{i},a_{i},t_{i} \rbrace$ be the copy of $\widehat{\mathcal{A}}_{i}^{\beta}$ in $T$, and let $\widetilde{Q_{i}}$ be the copy of $Q_{i}$ in $T$. Let $\alpha$ be the ordering of $A = \bigcup_{i=1}^{l}(V(\widetilde{\mathcal{A}}_{i}^{\beta})\cup V(\widetilde{Q}_{i}))$ according to their appearance in $(S_{1},...,S_{\mid w \mid})$ (that is for $v,v'\in A$, if $\xi (v)<\xi (v')$, then $v$ appears before $v'$ in $\alpha$). Let $ i \in [l]$ such that $\mathcal{A}_{i}^{\beta}$ is a left $\beta_{1}$-asteroid. If $p_{i} \leftarrow a_{i}$, then we remove $c_{i},r_{i},q_{i},o_{i},s_{i},t_{i}$ from $\alpha$. Otherwise, if $q_{i} \leftarrow s_{i}$, then we remove $d_{i},g_{i},n_{i},r_{i},p_{i},a_{i}$ from $\alpha$. Otherwise, if $r_{i} \leftarrow t_{i}$, then we remove $d_{i},g_{i},n_{i},p_{i},q_{i},a_{i}$ from $\alpha$. Otherwise, $p_{i} \rightarrow a_{i}$, $q_{i} \rightarrow s_{i}$ and $r_{i} \rightarrow t_{i}$, in this case we remove $c_{i},d_{i},b_{i},n_{i},o_{i},f_{i}$ from $\alpha$. Note that in the first three cases we obtain the cyclic ordering of the left $\beta_{1}$-asteroid and in the last case we obtain the forest ordering of the left $\beta_{1}$-asteroid. Let $ i \in [l]$ such that $\mathcal{A}_{i}^{\beta}$ is a left $\beta_{2}$-asteroid. If $p_{i} \leftarrow a_{i}$, then we remove $c_{i},r_{i},q_{i},o_{i},s_{i},f_{i}$ from $\alpha$. Otherwise, if $q_{i} \leftarrow s_{i}$, then we remove $d_{i},g_{i},n_{i},r_{i},p_{i},a_{i}$ from $\alpha$. Otherwise, if $r_{i} \leftarrow f_{i}$, then we remove $d_{i},g_{i},n_{i},p_{i},q_{i},a_{i}$ from $\alpha$. Otherwise, $p_{i} \rightarrow a_{i}$, $q_{i} \rightarrow s_{i}$ and $r_{i} \rightarrow f_{i}$, in this case we remove $c_{i},d_{i},b_{i},n_{i},o_{i},t_{i}$ from $\alpha$. Note that in the first three cases we obtain the cyclic ordering of the left $\beta_{2}$-asteroid and in the last case we obtain the forest ordering of the left $\beta_{2}$-asteroid. Let $  i \in [l]$ such that $\mathcal{A}_{i}^{\beta}$ is a right $\beta_{1}$-asteroid. If $d_{i} \leftarrow p_{i}$, then we remove $c_{i},g_{i},n_{i},r_{i},q_{i},t_{i}$ from $\alpha$. Otherwise, if $g_{i} \leftarrow r_{i}$, then we remove $d_{i},p_{i},q_{i},o_{i},s_{i},a_{i}$ from $\alpha$. Otherwise, if $c_{i} \leftarrow q_{i}$, then we remove $d_{i},r_{i},p_{i},o_{i},s_{i},a_{i}$ from $\alpha$. Otherwise, $d_{i} \rightarrow p_{i}$, $g_{i} \rightarrow r_{i}$ and $c_{i} \rightarrow q_{i}$, in this case we remove $b_{i},n_{i},o_{i},f_{i},a_{i},t_{i}$ from $\alpha$. Note that in the first three cases we obtain the cyclic ordering of the right $\beta_{1}$-asteroid and in the last case we obtain the forest ordering of the right $\beta_{1}$-asteroid. Let $ i \in [l]$ such that $\mathcal{A}_{i}^{\beta}$ is a right $\beta_{2}$-asteroid. If $d_{i} \leftarrow p_{i}$, then we remove $b_{i},g_{i},n_{i},r_{i},q_{i},t_{i}$ from $\alpha$. Otherwise, if $g_{i} \leftarrow r_{i}$, then we remove $d_{i},p_{i},q_{i},o_{i},s_{i},a_{i}$ from $\alpha$. Otherwise, if $b_{i} \leftarrow q_{i}$, then we remove $d_{i},r_{i},p_{i},o_{i},s_{i},a_{i}$ from $\alpha$. Otherwise, $d_{i} \rightarrow p_{i}$, $g_{i} \rightarrow r_{i}$ and $b_{i} \rightarrow q_{i}$, in this case we remove $c_{i},n_{i},o_{i},f_{i},a_{i},t_{i}$ from $\alpha$. Note that in the first three cases we obtain the cyclic ordering of the right $\beta_{2}$-asteroid and in the last case we obtain the forest ordering of the right $\beta_{2}$-asteroid.  We apply this rule  for all $ i \in [l]$. We obtain one of the orderings in $\Theta$ of $V(H)$. So $T$ contains $H$. $\hfill{\square}$  
\end{proof}
\vspace{3mm}\\
Now we are ready to prove Theorem \ref{j}:\vspace{3mm}\\
\noindent \sl {Proof of Theorem \ref{j}.} \upshape
Let $H$ be an asterism under an ordering $\theta$ of its vertices. We may assume that $H$ is a regular asterism since every asterism is a subtournament of a regular asterism. Let $\mathcal{A}^{\beta}_{1},...,\mathcal{A}^{\beta}_{l}$ be the $\beta$-asteroids of $H$ under $\theta$, and let $Q_{1},...,Q_{l}$ be the stars of $H$$\mid$$(V(H)\backslash \bigcup_{i=1}^{l}V(\mathcal{A}^{\beta}_{i}))$ under $\overline{\theta}$. Let $\epsilon > 0$ be small enough and let $\lambda > 0$ be small enough. Assume that $H$ does not satisfy $EHC$, then there exists an $H$-free $\epsilon$-critical tournament $T$. By Theorem \ref{i}, $T$ contains a smooth $(c,\lambda ,w)$-structure $(S_{1},...,S_{\mid w \mid})$ corresponding to $\widehat{H}$ under $\widehat{\theta}$ for some $c >0$ and appropriate $\lbrace 0,1 \rbrace$-vector $w$. Then by the previous corollary, $T$ contains $H$, a contradiction. This terminates the proof. $\hfill{\square}$
\begin{figure}[h]
	\centering
	\includegraphics[width=0.8\linewidth]{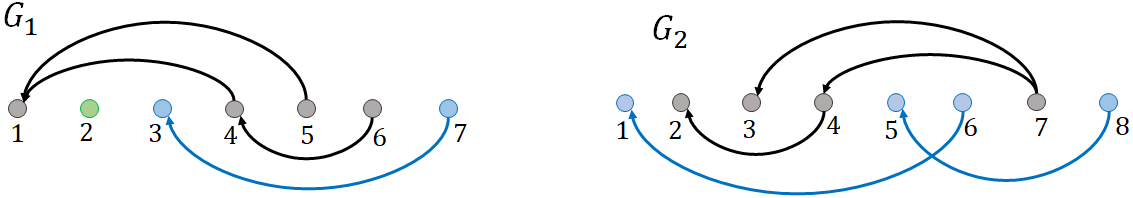}
	\caption{Subtournaments of asterisms. Arcs corresponding to four-paths, stars, and singletons are colored in black, blue, and green respectively.  All the nondrawn arcs are forward.}
	\label{fig:subasterisms}
\end{figure}

We have infinitely many tournaments that are subtournaments of asterisms and does not belong to a class proved to satisfy EHC (see Figure \ref{fig:subasterisms}). By the following corollary, such tournaments also satisfy EHC:
\begin{corollary}
Every subtournament of an asterism satisfies EHC.
\end{corollary}
\begin{proof}
The result follows from Theorem \ref{j} and from the fact that the Erd\"{o}s-Hajnal property is a hereditary property. $\hfill{\square}$
\end{proof}
\section{Galaxies with Spiders}\label{secgalaxywithspiders}
\hspace{3.5mm}   Unfortunately the following is still a conjecture:
\begin{conjecture} \label{nebula}
Every nebula satisfies the Erd\"{o}s-Hajnal conjecture.
\end{conjecture}

In \cite{polll} Berger et al. proved Conjecture \ref{nebula} for every galaxy. In galaxies there are no middle stars, and the centers of the stars are not allowed to appear between the leaves of some star. Then in \cite{kg}, Choromanski confirmed Conjecture \ref{nebula} for every constellation. In constellations the condition concerning the centers of the stars in galaxies is relaxed, but middle stars are still not allowed. In this section we confirm Conjecture \ref{nebula} for every galaxy with spiders.
The family of galaxies with spiders contains infinitely many tournaments that are neither galaxies nor constellations. In galaxies with spiders there are "\textit{middle stars}" and there are "centers of stars between leaves of another stars" under some conditions that we will explain in the following subsection (this condition is different from that in constellations).  
The technique of the proof is different than that used in galaxies and constellations. Also it is different than that used  in asterisms. This technique does not depend on corresponding digraphs as in the above section. We will first define the infinite family of tournaments $-$ the family of so$-$called galaxies with spiders and we will introduce several technical definitions.
\subsection{Definitions}
\hspace{3.5mm} A tournament $\mathcal{S}$ on $n$ vertices with $V(\mathcal{S})= \lbrace 1,...,m,m+1,...,m+4,m+5,m+6,...,n\rbrace$ is a \textit{middle $spider_{1}$} (resp. \textit{middle $spider_{2}$}) if there exist an ordering $\theta = (1,...,m,m+1,...,m+4,m+5,m+6,...,n)$ of its vertices such that the set of backward arcs of $\mathcal{S}$ under $\theta$ is $\lbrace(m+4,m+1),(m+5,m+2)\rbrace \cup\lbrace(m+1,i): i=1,...,m \rbrace \cup \lbrace(i,m+5): i= m+6,...,n\rbrace$ (resp. $\lbrace(m+4,m+1),(m+5,m+2)\rbrace \cup\lbrace(m+5,i): i=1,...,m \rbrace \cup \lbrace(i,m+1): i= m+6,...,n\rbrace$) (see example in Figure \ref{fig:spider}). We write $\mathcal{S}= \lbrace 1,...,m,m+1,...,m+4,m+5,m+6,...,n\rbrace$ and we call $\theta$ a \textit{middle spider$_{1}$} (resp. \textit{middle spider$_{2}$}) \textit{ordering} of $\mathcal{S}$, $1,...,m,m+6,...,n$ the \textit{legs} of $\mathcal{S}$, $m+1,...,m+5$ the \textit{interoir vertices} of $\mathcal{S}$, $m+2,m+3,m+4$ the \textit{petals} of $\mathcal{S}$, and $m+1,m+5$ the \textit{centers} of $\mathcal{S}$. A \textit{middle spider} is a \textit{middle $spider_{1}$} or a \textit{middle $spider_{2}$}. A \textit{middle spider ordering} is a middle spider$_{1}$ ordering or a middle spider$_{2}$ ordering. A tournament $\mathcal{S}$ on $n$ vertices with $V(\mathcal{S})= \lbrace 1,...,n-4,...,n\rbrace$ is a \textit{right spider} if there exist an ordering $\theta = (1,...,n-4,...,n)$ of its vertices such that the set of backward arcs of $\mathcal{S}$ under $\theta$ is $\lbrace(n,n-3),(n-1,n-4)\rbrace \cup\lbrace(n,i): i\in X_{1} \rbrace \cup \lbrace(n-4,i): i\in X_{2}\rbrace$, where $X_{1}\cup X_{2}=\lbrace 1,...,n-5\rbrace$ and $X_{1}\cap X_{2}=\phi$ (see example in Figure \ref{fig:spider}). We write $\mathcal{S}= \lbrace 1,...,n-4,...,n\rbrace$ and we call $\theta$ a \textit{right spider ordering} of $\mathcal{S}$, $1,...,n-5$ the \textit{legs} of $\mathcal{S}$, $n-4,...,n$ the \textit{interoir vertices} of $\mathcal{S}$, $n-3,n-2,n-1$ the \textit{petals} of $\mathcal{S}$, and $n-4,n$ the \textit{centers} of $\mathcal{S}$.  A tournament $\mathcal{S}$ on $n$ vertices with $V(\mathcal{S})= \lbrace 1,2,3,4,5,...,n\rbrace$ is a \textit{left spider} if there exist an ordering $\theta = (1,2,3,4,5,...,n)$ of its vertices such that the set of backward arcs of $\mathcal{S}$ under $\theta$ is $\lbrace(4,1),(5,2)\rbrace \cup\lbrace(i,5): i\in X_{1} \rbrace \cup \lbrace(i,1): i\in X_{2}\rbrace$, where $X_{1}\cup X_{2}=\lbrace 6,...,n\rbrace$ and $X_{1}\cap X_{2}=\phi$ (see example in Figure \ref{fig:spider}). We write $\mathcal{S}= \lbrace 1,...,5,6,...,n\rbrace$ and we call $\theta$ a \textit{left spider ordering} of $\mathcal{S}$, $6,...,n$ the \textit{legs} of $\mathcal{S}$, $1,...,5$ the \textit{interoir vertices} of $\mathcal{S}$, $2,3,4$ the \textit{petals} of $\mathcal{S}$, and $1,5$ the \textit{centers} of $\mathcal{S}$. A \textit{spider} is a middle spider or a left spider or a right spider. A \textit{spider ordering} is a middle spider ordering or a left spider ordering or a right spider ordering.
\begin{figure}[h]
	\centering
	\includegraphics[width=0.8\linewidth]{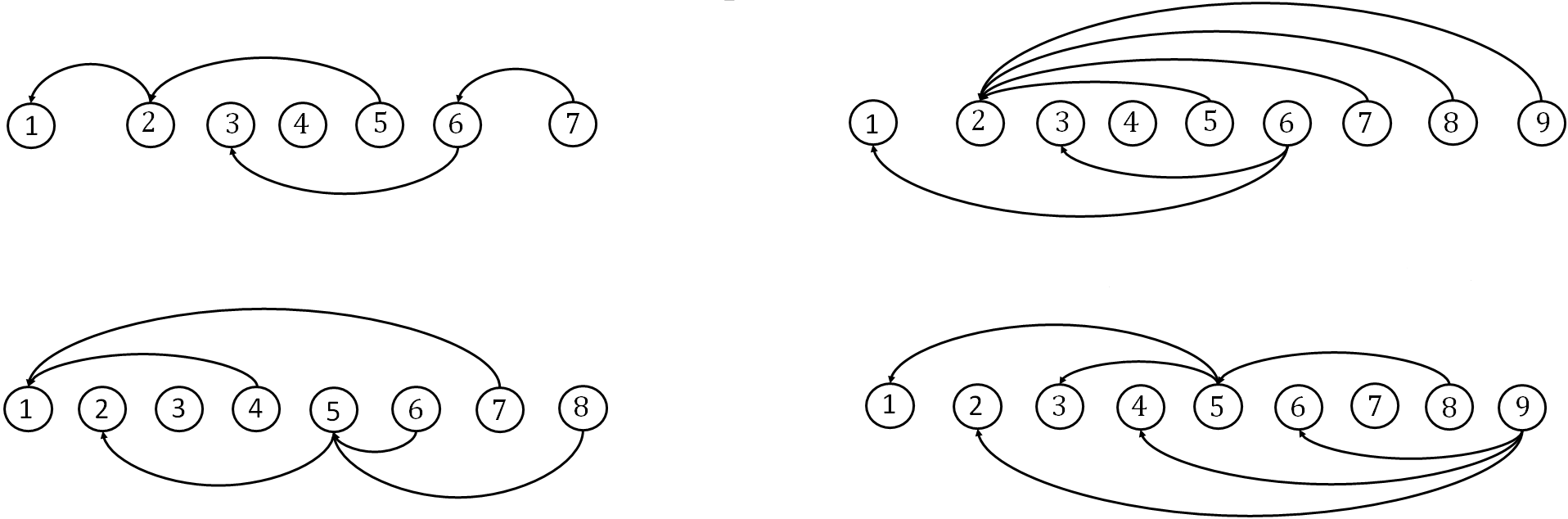}
	\caption{Middle spiders on the top, left spider and right spider on the bottom. All spiders are drawn under their spider ordering. All the nondrawn arcs are forward.}
	\label{fig:spider}
\end{figure}

Let $ \theta = (v_{1},...,v_{n}) $ be an ordering of the vertex set $V(T)$ of an $ n- $vertex tournament $T$. A \textit{spider} $\mathcal{S}=\lbrace v_{i_{1}},...,v_{i_{t}}\rbrace$ \textit{of $T$ under $\theta$} (where $i_{1}<...<i_{t}$) is a subtournament of $T$ induced by $\lbrace v_{i_{1}},...,v_{i_{t}}\rbrace$ such that $\mathcal{S}$ is a spider, $\mathcal{S}$ has the spider ordering $ (v_{i_{1}},...,v_{i_{t}})$ under $\theta$, and the interior vertices of $\mathcal{S}$ are consecutive under $\theta$.
 A \textit{triangle} $C =\lbrace v_{i_{1}},v_{i_{2}},v_{i_{3}}\rbrace$ \textit{of $T$ under $\theta$} is a transitive subtournament of $T$ induced by $\lbrace v_{i_{1}},v_{i_{2}},v_{i_{3}}\rbrace$ such that $(v_{i_{3}},v_{i_{2}},v_{i_{1}})$ is its transitive ordering (i.e $v_{i_{1}}\leftarrow v_{i_{2}}$ and $\lbrace v_{i_{1}},v_{i_{2}}\rbrace\leftarrow v_{i_{3}}$), and $i_{1}<i_{2}<i_{3}$. We can choose arbitrary any one of the vertices of $C$ to be the \textit{center of $C$}, and the other vertices will be called the \textit{leaves of $C$}. 

\vspace{3mm} Let $H$ be a tournament such that there exists an ordering $\theta$ of its vertices such that $V(H)$ is the disjoint union of $V(\mathcal{S}_{1}),...,V(\mathcal{S}_{l}),X$ where $\mathcal{S}_{1},...,\mathcal{S}_{l}$ are the spiders of $H$ under $\theta$, $H$$\mid$$X$ is a galaxy under $\overline{\theta}$ where $\overline{\theta}$ is the restriction of $\theta$ to $X$ (let $Q_{1},...,Q_{m}$ be the stars of $H$$\mid$$X$ under $\overline{\theta}$), no interior vertex of spider appears in the ordering $\theta$ between leaves of a star of $H$$\mid$$X$ under $\overline{\theta}$, no center of a star of $H$$\mid$$X$ under $\overline{\theta}$ appears in the ordering $\theta$ between legs of a spider of $H$ under $\theta$ incident to the same center of this spider, and for any given $1\leq i<j\leq l$ every vertex of $V(\mathcal{S}_{i})$ is before every vertex of $V(\mathcal{S}_{j})$ under $\theta$. We call this ordering a \textit{super galaxy ordering of $H$}.\\ The \textit{contracting graph} of $H$ under $\theta$ is the undirected graph with vertices $N_{1},...,N_{f}$, where $N_{1},...,N_{f_{1}}$ are the sets of legs of the spiders $\mathcal{S}_{1},...,\mathcal{S}_{l}$ incident to the same center, $N_{f_{1}+1},...,N_{f_{2}}$ are the sets of petals of the spiders $\mathcal{S}_{1},...,\mathcal{S}_{l}$, and $N_{f_{2}+1},...,N_{f}$ are the sets of leaves of the stars $Q_{1},...,Q_{m}$ respectively. Two vertices $N_{i}$ and $N_{j}$ are adjacent if there exist a point $v\in N_{i}$ such that $v$ is between the left point and the right point of $N_{j}$, or there exist a point $v\in N_{j}$ such that $v$ is between the left point and the right point of $N_{i}$. Note that the vertices corresponding to $N_{i}$ are called the \textit{points} of $N_{i}$ for $i=1,...,f$. For each connected component $C$ of the contracting graph of $H$ under $\theta$ denote by $\mathcal{Z}(C)$ the union of subsets of $V(H)$ correspoding to its vertices (i.e $\mathcal{Z}(C)= \lbrace v\in V(H); v$ is a point of $ N_{i}$ for some $N_{i}\in V(C)\rbrace$). Let $C_{1},...,C_{k}$ be the connected components of the contracting graph. Denote by $(M_{1},...,M_{k})$  the ordering of the sets $\mathcal{Z}(C_{1}),...,\mathcal{Z}(C_{k})$ such that for any given $1\leq i < j\leq k$ every vertex of $M_{i}$ is before every vertex of $M_{j}$. Notice that $P_{\theta}(H)=\lbrace M_{1},...,M_{k}\rbrace$ form a partition of $V(H)\backslash \mathcal{C}$, where $\mathcal{C}$ is the set of centers of $\mathcal{S}_{1},...,\mathcal{S}_{l},Q_{1},...,Q_{m}$.
   
\vspace{2mm} A tournament $T$ is a \textit{galaxy with spiders} if there exists a super galaxy ordering $\theta$ of its vertices satisfying the following: let $v$ and $v'$ be legs of distinct spiders $\mathcal{S}_{i}$ and $\mathcal{S}_{j}$, and suppose that $v$ is in some $M\in P_{\theta}(H)$. If there is a center of a star of $T$$\mid$$X$ under $\overline{\theta}$ lying between a center of $\mathcal{S}_{i}$ and a center of $\mathcal{S}_{j}$ under $\theta$ then $v' \notin M$. In this case we call $\theta$ a \textit{galaxy with spiders ordering of} $T$ (alternatively: $T$ is a \textit{galaxy with spiders under} $\theta$). If $T$$\mid$$X$ is a regular galaxy under $\overline{\theta}$, then $T$ is called a \textit{regular galaxy with spiders} (see Figure \ref{fig:gspider}). If $X=\phi$ then $T$ is called a \textit{clutter}.

\vspace{2.5mm}It can be clearly noticed that the family of galaxies with spiders is a larger family of tournaments than galaxies. So this generalizes the result in \cite{polll}.  Also it can be noticed that the family of galaxies with spiders contains infinitely
many tournaments that are neither galaxies nor constellations. That's because in galaxies with spiders we have middle stars, while in galaxies and constellations middle stars are not allowed, and because the relaxed condition concerning the centers of stars is different than that in constellations. In what follows we give examples of tournaments having middle stars and centers of stars between leaves of another stars, that belongs to the family of galaxies with spiders and are neither galaxies nor constellations. We will start by the smallest tournament with respect to the number of vertices, denoted by $H_1$, that belongs to the family of galaxies with spiders.\vspace{2.5mm}\\
- $H_1$ is a seven-vertex tournament. It is the smallest because all the tournaments having star ordering and  with at most six vertices are galaxies \cite{polll,bnmm}. $\{(6,3),(3,1),(7,2),(7,4)\}$ is the set of backward arcs of $H_1$ under the
ordering $\alpha_{1}=(1, ..., 7)$. Clearly one can notice that $H_1$  is a right spider under this ordering, and so $H_1$ is a galaxy with spiders (precisely, a clutter). The tournament $H_1$ drawn under the ordering $\alpha_{1}$  consists of a middle star $\{1,3,6\}$, a right star $\{2,4,7\}$, and a singleton. Notice that the center of the middle star appears in the ordering $\alpha_1$ between the leaves of the right star. Also one can check that $(1, ..., 7)$ is the unique ordering of the vertex set of  $H_1$ is which we have stars or middle stars (see Figure \ref{fig:gspider}).\vspace{2.5mm}\\
- $H_2$ is a nine-vertex tournament, such that $\{(3,1),(6,3),(7,4),(9,7),(8,2)\}$ is the set of backward arcs under $\alpha_{2}=(1,...,9)$. $H_2$ consists of two middle stars: $\{1,3,6\}$ and $\{4,7,9\}$, one left star: $\{2,8\}$, and a singleton $\{5\}$. In other words, $H_2$ consists of one middle spider: $\{1,3,4,5,6,7,9\}$ and one left star: $\{2,8\}$. One can check that $\alpha_{2}$ is the unique nebula ordering of $H_2$. So $H_2$ is a galaxy with spiders, and $H_2$ is neither a constellation nor a galaxy (see Figure \ref{fig:gspider}).\vspace{2.5mm}\\
- Take the tournament $H_3$ with vertex set $\{1,...,13\}$, such that $H_3$ is the disjoint union of the right star $\{1,4,7\}$, and the three left stars: $\{2,8\}$, $\{3,6,10,12\}$, and $\{9,11,13\}$ under the ordering $\alpha_{3}=(1,...,13)$ of its vertices. In other words $H_3$ consists of the middle spider $\{1,3,4,5,6,7,10,12\}$ and the two stars: $\{2,8\}$ and $\{9,11,13\}$. Clearly $\alpha_{3}$ is not a constellation ordering because the center and one of the leaves of the left star $\{9,11,13\}$ appears in the ordering $\alpha_{3}$ between the leaves of the left star $\{3,6,10,12\}$. Also one can check that $\alpha_{3}$ is the only nebula ordering of $H_3$. So $H_3$ is a galaxy with spiders, and neither a constellation nor a galaxy (see Figure \ref{fig:gspider}).\vspace{2.5mm}\\
- The connected components of the graph of backward arcs of the $15$-vertex  tournament $H_4$ under $\alpha_{4}=(1,...,15)$ are: the middle star  $\{2,5,8,11\}$ with center $5$, the right star $\{9,13,15\}$, and the two left stars $\{1,4,7,10\}$ and $\{6,12,14\}$. In other words $H_4$ consists of the left spider $\{1,...,5,7,8,10,11\}$ and the two stars: $\{9,13,15\}$ and $\{6,12,14\}$. Clearly $\alpha_{4}$ is not a constellation ordering. Also one can check that $\alpha_{4}$ is the only nebula ordering of $H_4$. So $H_4$ is a galaxy with spiders, and neither a constellation nor a galaxy (see Figure \ref{fig:gspider}).\vspace{2.5mm}\\
- $H_5$ is a galaxy with spiders, with vertex set $\{1,...,36\}$, such that under the ordering $\alpha_{5}:=(1,...,36)$, $H_5$ is the disjoint union of the four middle stars: $\{2,5,8\}$, $\{12,15,18\}$, $\{16,19,21,23\}$, and $\{24,30,31,34\}$, the three right star: $\{7,9,13\}$, $\{22,25,36\}$, and $\{27,29,32,35\}$,  the three left stars: $\{1,4,6,10\}$, $\{11,26,28\}$, and $\{14,20\}$, and the singletons $\{3,17,33\}$. Also $H_5$ can be viewed under $\alpha_{5}$ as the disjoint union of the the middle spider $\{12,15,...,19,21,23\}$, the left spider $\{1,...,6,8,10\}$, the right spider $\{24,27,29,...,35\}$, the two left stars: $\{11,26,28\}$ and $\{14,20\}$, and the two right stars: $\{7,9,13\}$ and $\{22,25,36\}$. Under $\alpha_{5}$ we have three middle stars and centers of stars appearing between leaves of stars. Moreover there exists no constellation ordering of $H_5$ (we leave this for the reader). We conclude that $H_5$ is a galaxy with spiders as $\alpha_{5}$ satisfies all the conditions of galaxies with spiders, and neither a constellation nor a galaxy (see Figure \ref{fig:gspider}).\\
\begin{figure}[h]
	\centering
	\includegraphics[width=0.9\linewidth]{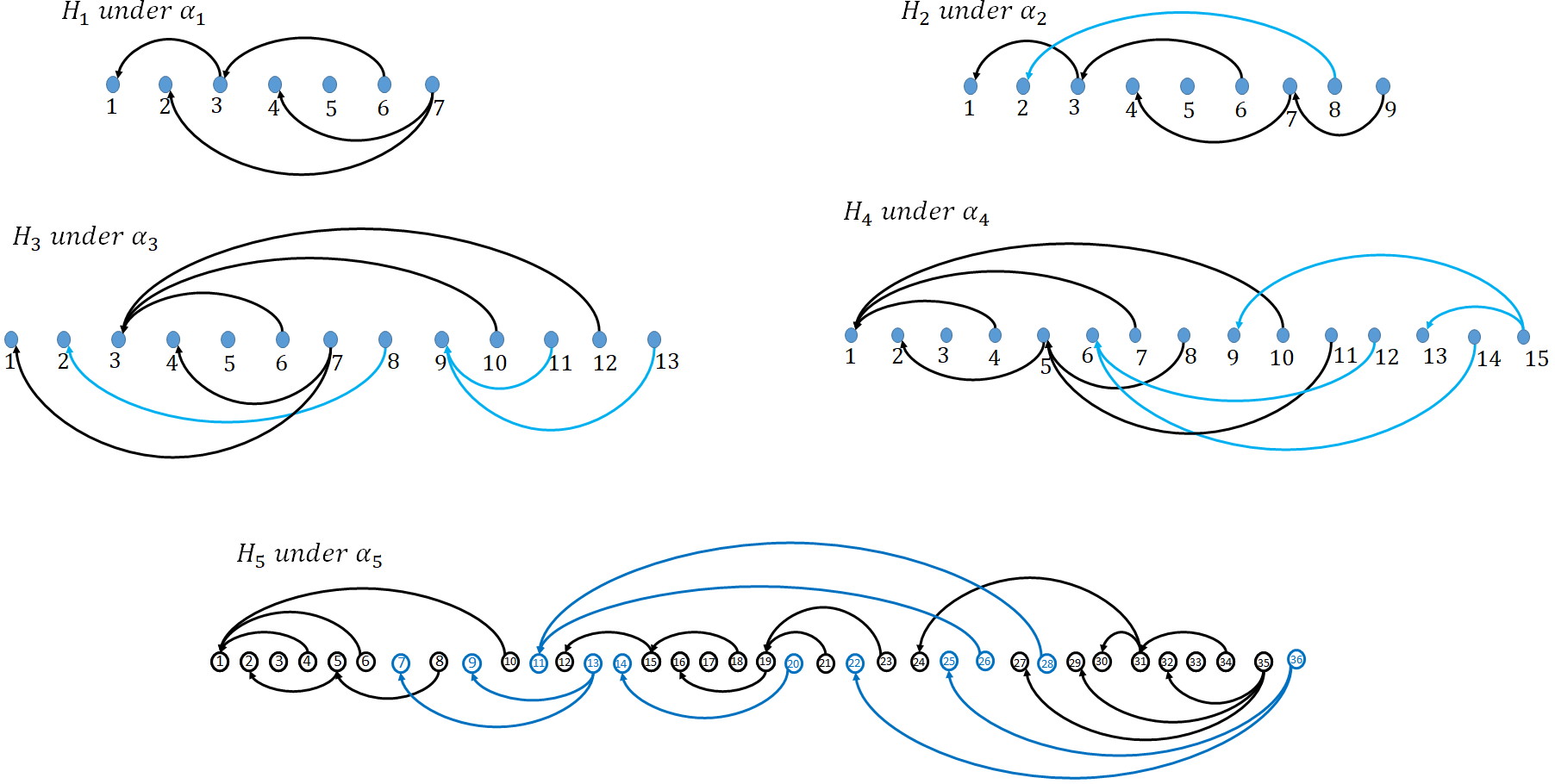}
	\caption{Five galaxies with spiders. Spiders are colored in black and the rest of the stars are colored in blue. All the nondrawn arcs are forward.}
	\label{fig:gspider}
\end{figure}    
\\The main result in this section is:
\begin{theorem} \label{p}
Every galaxy with spiders satisfy the Erd\"{o}s-Hajnal Conjecture.
\end{theorem}

\begin{definition}\normalfont Let $G$ be a tournament with five vertices $v_{1},...,v_{5}$, and let $\theta_{1}=(v_{1},v_{2},v_{3},v_{4},v_{5})$ be an ordering of $V(G)$. Operation $\alpha$ is the permutation of the vertices $v_{1},...,v_{5}$ that converts the ordering $\theta_{1}$ to the ordering $\theta_{2}=(v_{4},v_{1},v_{3},v_{5},v_{2})$ of $V(G)$.
\end{definition}

Let $\mathcal{S}=\lbrace 1,...,n\rbrace$ be a spider and let $\theta$ be its spider ordering. Define the \textit{cyclic ordering of} $\mathcal{S}$ to be the ordering of $V(\mathcal{S})$ obtained from $\theta$ by performing operation $\alpha$ to the interoir vertices of $\mathcal{S}$ under $\theta$. Denote this ordering by $\vartheta$. In this case $V(\mathcal{S})$ is the disjoint union of $V(S_{1}),V(S_{2}),V(C)$, where $S_{1}$ and $S_{2}$ are the stars of $\mathcal{S}$ under $\vartheta$, and $C$ is the triangle of $\mathcal{S}$ under $\vartheta$.   
In Figure \ref{fig:cyclic} a ten-vertex middle spider $S= \lbrace 1,...,m,m+1,...,m+4,m+5,m+6,...,n\rbrace$ is drawn under its cyclic ordering (in this example we have $m=3$ and $n=10$).  
\begin{figure}[h]
	\centering
	\includegraphics[width=0.5\linewidth]{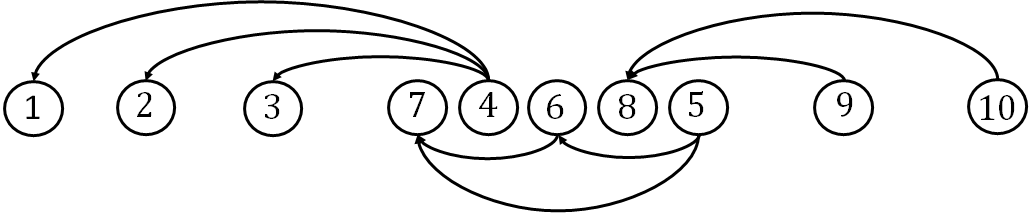}
	\caption{A middle spider$_{1}$ $S= \lbrace 1,2,3,4,5,6,7,8,9,10\rbrace$ drawn under its cyclic ordering. All the nondrawn arcs are forward.}
	\label{fig:cyclic}
\end{figure} 

Let $\mathcal{S}= \lbrace 1,...,n\rbrace$ be a spider. A \textit{mutant spider} $\tilde{\mathcal{S}}= \lbrace 1,...,n\rbrace$ is obtained from $\mathcal{S}$ by deleting some arcs as follows: If $\mathcal{S}= \lbrace 1,...,m,m+1,...,m+4,m+5,m+6,...,n\rbrace$ is a middle spider$_{1}$ (resp. middle spider$_{2}$) then  $\tilde{\mathcal{S}}$ is obtained from $\mathcal{S}$ by deleting the arcs $(i,m+4)$ (resp. $(i,m+2)$) for all $i\in\lbrace 1,...,m\rbrace$ and the arcs $(m+2,i)$ (resp. $(m+4,i)$) for all $i\in \lbrace m+6,...,n\rbrace$. If $\mathcal{S}= \lbrace 1,...,n-4,...,n\rbrace$ is a right spider then  $\tilde{\mathcal{S}}$ is obtained from $\mathcal{S}$ by deleting the arcs $(i,n-3)$ for all $i\in X_{1}$ and the arcs $(i,n-1)$ for all $i\in X_{2}$, where $X_{1}$ (resp. $X_{2}$) is the set of all outneighbors of $n$ (resp. $n-4$) in $\lbrace 1,...,n-5\rbrace$. If $\mathcal{S}= \lbrace 1,...,5,...,n\rbrace$ is a left spider then  $\tilde{\mathcal{S}}$ is obtained from $\mathcal{S}$ by deleting the arcs $(2,i)$ for all $i\in X_{1}$ and the arcs $(4,i)$ for all $i\in X_{2}$, where $X_{1}$ (resp. $X_{2}$) is the set of all inneighbors of $5$ (resp. $1$) in $\lbrace 6,...,n\rbrace$. A \textit{mutant clutter} $\tilde{H}$ under $\theta$ is obtained from a clutter $H$ under $\theta$ by transforming all spiders of $H$ under $\theta$ to mutant spiders.

\begin{remark} Unlike galaxies and constellations, depending only on the galaxy with spiders ordering $\theta$ of a galaxy with spiders $H$ in proving the conjecture for $H$ will not work (see subsection \ref{overview}). To this end, we studied the structure of these tournaments. In our proof we will use several crucial orderings for  a galaxy with spiders $H$. These orderings are outcomes of applying some vertex permutation on $\theta$.\end{remark}

  Let $H$ be a regular galaxy with spiders under an ordering $\theta = (v_{1},...,v_{\mid H \mid})$ of its vertices with $\mid$$H$$\mid$ $= h$. Let $\mathcal{S}_{1},...,\mathcal{S}_{l}$ be the spiders of $H$ under $\theta$. Define $\Theta_{\theta}(H) = \lbrace \theta^{'}$ an ordering of $V(H)$: $\theta^{'}$ is obtained from $\theta$ by performing operation $\alpha$ to the interior vertices of some spiders of $H$ under $\theta$$\rbrace$. Notice that $\mid$$\Theta$$\mid$ $= 2^{l}$. 
  \begin{remark} The spider is transformed into a triangle and two stars under applying operation $\alpha$ to its interior vertices, which gives nice structural backward arc configurations for galaxieas with spiders.\end{remark}
  
   Our proof will heavily depend on these special orderings, as in all possible cases, one of the orderings in  $\Theta_{\theta}(H)$ will save us (this will be discussed informally in subsection \ref{overview} and formally in sebsection \ref{proof}). In Figure \ref{fig:orderingsH5}, for every $\theta^{'}\in \Theta_{\theta}(H_{5}):=\{\theta ,\theta_{1},...,\theta_{7}\}$, we draw the tournament $H_5$ under $\theta^{'}$, where $\theta :=\alpha_{5}$.   
\begin{figure}[h]
	\centering
	\includegraphics[width=1\linewidth]{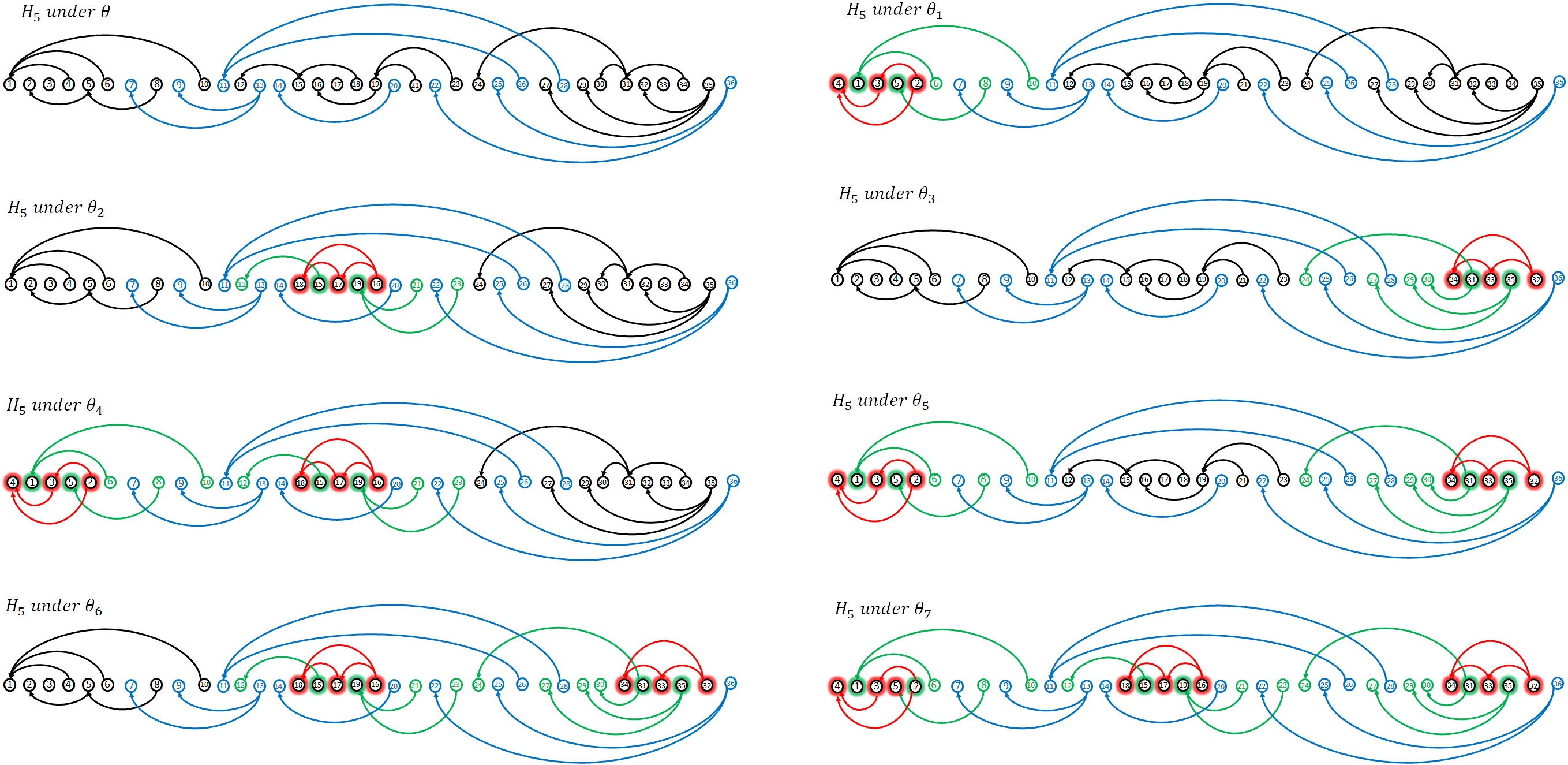}
	\caption{For all $\theta^{'}\in \Theta_{\theta}(H_5)$, $H_5$ is drawn under $\theta^{'}$. Spiders are colored in black, stars are colored in blue. Triangles and stars produced under performing operation $\alpha$ are colored in red and green respectively. All the non drawn arcs are forward.}
	\label{fig:orderingsH5}
\end{figure}

In what follows we introduce the notion of $\{s,t\}$-vectors and the sequence of associated structures that will be crucial in our latter analysis. $\{s,t\}$-vectors are marker vectors that plays a crucial role in determining which ordering $\theta^{'}\in \Theta_{\theta}(H)$ will save use, where $s$ and $t$ stands for spiders and triangles respectively. Then for each ordering $\theta^{'}$ of $H$ we will introduce a sequence of associated structures. In the proof we will prove that if a tournament $T$ does not contain a transitive subtournament with polynomial size, then it contains a copy of $H$ under $\theta^{'}$, for some $\theta^{'}\in \Theta_{\theta}(H)$. The method is to start with a smooth $(c,\lambda ,w)$-structure, with $w=(0,1,0,1,...,0,1)$ and $\mid$$w$$\mid$ is large enough. Then depending on the outcomes, we will be able to extract the ordering $\theta^{'}\in \Theta_{\theta}(H)$ and then an associated smooth $(c',\lambda^{'},w')$-structure from the sequence of associated structures corresponding to $H$ under $\theta^{'}$ that will serve us.   \vspace{2mm}

 Let $H$ be a regular galaxy with spiders under an ordering $\theta = (v_{1},...,v_{h})$ of its vertices with $\mid$$H$$\mid$ $= h$. Let $\theta^{'}= (u_{1},...,u_{h})\in \Theta_{\theta}(H)$ and denote by $(\mathcal{K}_{1},...,\mathcal{K}_{l})$ the ordering of  the spiders and triangles of $H$ under $\theta^{'}$ such that for any given $1\leq i < j\leq l$ every vertex of $\mathcal{K}_{i}$ is before every vertex of $\mathcal{K}_{j}$. Define $g^{H,\theta^{'}}$ to be the $\lbrace \mathsf{s},\mathsf{t} \rbrace$$-$vector such that for $i=1,...,l$, $g^{H,\theta^{'}}(i) = \mathsf{s}$ if and only if $\mathcal{K}_{i}$ is a spider of $H$ under $\theta^{'}$ and $g^{H,\theta^{'}}(i) = \mathsf{t}$ otherwise. We say that an $\lbrace \mathsf{s},\mathsf{t} \rbrace$$-$vector $g$ \textit{corresponds to $H$ under $\theta^{'}$} if $g=g^{H,\theta^{'}}$.\vspace{2.5mm}

Now we extend the definition of corresponding structures of a tournament under an ordering of its vertices (defined in \cite{polll}) to a sequence of associated structures under partitioning (corresponding structure is one of the associated structures under partitioning), as follows:\vspace{1.5mm}\\
  Let $\mathcal{C}$ be any set of choosen centers of all the triangles, stars, and spiders (we have many choices of $\mathcal{C}$). Let ${}^{\mathcal{C}}s^{H,\theta^{'}}$ be a $\lbrace 0,1 \rbrace$$-$vector such that ${}^{\mathcal{C}}s^{H,\theta^{'}}(i) = 0$ if and only if $u_{i}\in \mathcal{C}$. Let $\mathcal{K}_{i_{1}},...,\mathcal{K}_{i_{t}}$ be the triangles of $H$ under $\theta^{'}$ and let $\theta^{'*}$ be the restriction of $\theta^{'}$ to $V=V(H)\backslash\bigcup_{j=1}^{t}V(\mathcal{K}_{i_{j}})$ (notice that $H$$\mid$$V$ is a galaxy with spiders under $\theta^{'*}$). Let $P_{\theta}(H)=\lbrace M_{1},...,M_{k}\rbrace$ and let $P_{\theta^{'*}}(H$$\mid$$V)=\lbrace N_{1},...,N_{z_{1}}\rbrace$. Let $U_{1},...,U_{z_{2}}$ denote the singleton sets containing the leaves of $\mathcal{K}_{i_{1}},...,\mathcal{K}_{i_{t}}$ (notice that $z_{2}=2i_{t}$). Denote the sets $N_{1},...,N_{z_{1}},U_{1},...,U_{z_{2}}$ as $\lbrace \mathcal{R}_{1},...,\mathcal{R}_{r}\rbrace$ with $r=z_{1}+z_{2}$ such that for any given $1\leq i < j\leq r$ every vertex of $\mathcal{R}_{i}$ is before every vertex of $\mathcal{R}_{j}$. Define $P_{\theta^{'}}(H)=\lbrace \mathcal{R}_{1},...,\mathcal{R}_{r}\rbrace$. Notice that if $\theta^{'}=\theta$ then $P_{\theta^{'}}(H)=\lbrace M_{1},...,M_{k}\rbrace$. Let $\mathcal{I}$ be the set of all possible partitions $I=\lbrace I_{1},...,I_{t_{I}}\rbrace \in \mathcal{I}$ of $\lbrace 1,...,r\rbrace$ such that $\forall 1\leq j\leq t_{I}$, the vertices of  $\bigcup_{i\in I_{j}}\mathcal{R}_{i}$ are consecutive under $\theta^{'}$. Let $I=\lbrace I_{1},...,I_{t_{I}}\rbrace \in \mathcal{I}$. Let $P_{I}=\lbrace \bigcup_{i\in I_{j}}\mathcal{R}_{i}: j=1,...,t_{I}\rbrace$. Define $\mathcal{P}=\lbrace P_{I}: I\in \mathcal{I}\rbrace$. \\  Let $P_{I}\in \mathcal{P}$. Let ${}^{\mathcal{C}}s^{H,\theta^{'}}_{P_{I}}$ be the $\lbrace 0,1\rbrace$$-$vector obtained from ${}^{\mathcal{C}}s^{H,\theta^{'}}$ by replacing every subsequence of consecutive $1's$ corresponding to vertices that belong to the same set in $P_{I}$ by single $1$. Let $\delta^{{}^{\mathcal{C}}s^{H,\theta^{'}}_{P_{I}}}:$ $\lbrace i: {}^{\mathcal{C}}s^{H,\theta^{'}}_{P_{I}} = 1 \rbrace \rightarrow \mathbb{N}$ be a function that assigns to every nonzero entry of ${}^{\mathcal{C}}s^{H,\theta^{'}}_{P_{I}}$ the number of consecutive $1'$s of ${}^{\mathcal{C}}s^{H,\theta^{'}}$ replaced by that entry of ${}^{\mathcal{C}}s^{H,\theta^{'}}_{P_{I}}$. We say that a smooth $(c,\lambda ,w)$$-$structure $\chi$ is \textit{galaxy with spiders-associated to $H$ under} $\theta^{'}$ if $w = {}^{\mathcal{C}}s^{H,\theta^{'}}_{P_{I}}$ for some $P_{I}\in \mathcal{P}$. If $\theta^{'}=\theta$ and $P_{I}=P_{\theta}(H)$ (i.e $w = {}^{\mathcal{C}}s^{H,\theta}_{P_{\theta}(H)}$) then $\chi$ is called \textit{galaxy with spiders-correlated to $H$ under} $\theta$ (in this case we only give $\chi$  another special name to distinguish it from the other smooth structures that are galaxy with spiders-associated to $H$ under $\theta^{'}$). In other words, $\chi_{\theta^{'}}(H):=\{\chi_{i} : \chi_{i}$ is a smooth $(c_{i},\lambda_{i},w_{i})$-structure such that $w_i = {}^{\mathcal{C}}s^{H,\theta^{'}}_{P_{I}}$ for some $P_{I}\in \mathcal{P}\}$ is the set of smooth structures that are \textit{galaxy with spiders-associated  to $H$ under} $\theta^{'}$.

Let $\theta^{'} = (u_{1},...,u_{h})\in \Theta$ and let $(S_{1},...,S_{\mid w \mid})$ be a smooth $(c,\lambda ,w)$$-$structure which is galaxy with spiders-associated to $H$ under the ordering $\theta^{'} $, and let $i_{r}$ be such that $w(i_{r}) = 1$. Assume that $S_{i_{r}} = \lbrace s^{1}_{i_{r}},...,s_{i_{r}}^{\mid S_{i_{r}} \mid} \rbrace$ and $(s^{1}_{i_{r}},...,s_{i_{r}}^{\mid S_{i_{r}} \mid})$ is a transitive ordering. Write $m(i_{r}) = \lfloor\frac{\mid S_{i_{r}} \mid}{\delta^{w}(i_{r})}\rfloor$.\\ Denote $S^{j}_{i_{r}} = \lbrace s^{(j-1)m(i_{r})+1}_{i_{r}},...,s_{i_{r}}^{jm(i_{r})} \rbrace$ for $j \in \lbrace 1,...,\delta^{w}(i_{r}) \rbrace$. For every $v \in S^{j}_{i_{r}}$ denote $\xi(v) = (\mid$$\lbrace k < i_{r}: w(k) = 0 \rbrace$$\mid$ $+$ $\displaystyle{\sum_{k < i_{r}: w(k) = 1}\delta^{w}(k) })$ $+$ $j$. For every $v \in S_{i_{r}}$ such that $w(i_{r}) = 0$ denote $\xi(v) = (\mid$$\lbrace k < i_{r}: w(k) = 0 \rbrace$$\mid$ $+$ $\displaystyle{\sum_{k < i_{r}: w(k) = 1}\delta^{w}(k) })$ $+$ $1$. We say that $H$ is \textit{well-contained in} $(S_{1},...,S_{\mid w \mid})$ which is galaxy with spiders-associated to $H$ under $\theta^{'} $ if there is an injective homomorphism $f$ of $H$ into $T$$\mid$$\bigcup_{i = 1}^{\mid w \mid}S_{i}$ such that $\xi(f(u_{j})) = j$ for every $j \in \lbrace 1,...,h \rbrace$.
\vspace{4.5mm}

Let $H$ be a galaxy with spiders under an ordering $\theta$ of its vertices. Let $\mathcal{S}_{1},...,\mathcal{S}_{l}$ be the spiders of $H$ under $\theta$, and let $Q_{1},...,Q_{m}$ be the stars of $H$$\mid$$X$ under $\overline{\theta}$, where $X=V(H)\backslash \bigcup_{i=1}^{l}V(\mathcal{S}_{i})$ and $\overline{\theta}$ is the restriction of $\theta$ to $X$. Let $\chi$ be a smooth $(c,\lambda ,w)$$-$structure which is galaxy with spiders-associated to $H$ under $\theta $.  If $X=\phi$ then $H$ is a clutter under $\theta$ and $\chi$ is called in this case \textit{clutter-associated to $H$ under $\theta$}. If $l=0$ then $H$ is a galaxy under $\theta$ and $\chi$ is called in this case \textit{galaxy-associated to $H$ under $\theta$}.

Let $K$ be a clutter under an ordering $\theta = (v_{1},...,v_{r})$ of its vertices and let $\chi =(S_{1},...,S_{\mid w \mid})$ be a smooth $(c,\lambda ,w)$$-$structure which is clutter-associated to $K$ under $\theta$. We say that the mutant clutter $\tilde{K}$ corresponding to $K$ under $\theta $ is \textit{well-contained in} $\chi$ if there is an injective homomorphism $f$ of $\tilde{K}$ into $T$$\mid$$\bigcup_{i = 1}^{\mid w \mid}S_{i}$ such that $\xi(f(v_{j})) = j$ for every $j \in \lbrace 1,...,r \rbrace$. 
\begin{remark}
In associated structures, the $1$'s that belong to the same set in chosen partition are replaced by a single $1$.
\end{remark}
\subsection{An overwiew of the proof}\label{overview}
\hspace{3.5mm} The proof of Theorem \ref{p} is very technical and a bit hard. Thus before
going into details we would like to give an intuition how the proof works, and explain main and most important steps of the proof. Moreover we will explain why the ideas/techniques used to prove the Erd\"{o}s-Hajnal conjecture for galaxies and constellations failed  to succeed with galaxies with spiders, and what are the new ideas and techniques that allows the breakthrough.\\
The proof is by contradiction. By taking $\epsilon >0$ small enough, we can assume the existence of a smooth $(c,\lambda ,w)$-structure in an $H$-free tournament $T$. Obviously the first step of the proof will be as in galaxies and constellations, because we get this step directly from Theorem \ref{i}  for any forbidden tournament $H$. This structure is a sequence of linear and transitive sets (see its definition in page $3$).  Before explaining the rest steps of the proof we would like to speak roughly what are galaxies and constellations, how the proof works for each of these two classes, and why using the same technique will not work in our setting.   In galaxies only left and right stars are allowed, and the proof uses only the galaxy ordering of the forbidden galaxy tournament. The proof that every galaxy satifies the Erd\"{o}s-Hajnal conjecture works as follows: linear sets are used to look for centers of the stars, and for every star, we look for its leaves in the same transitive set in order to get for free the right type of adjacency between the leaves. In galaxies we are able to look for leaves of the same star in the same transitive set because we don't have centers appearing between the leaves and we don't have middle stars. This is the reason behind making the proof for galaxies work.  From here one can know why proving the conjecture for nebulas is very hard. The difficulty in nebulas is that when having middle stars or centers of stars between leaves of another star, we are obliged to look for leaves of the same star in two or more different transitive sets, because we always look for centers in linear sets. This implies that the star we are looking for can't be necessarily found, because in this case we have no idea about the orientation of the arcs that links the fixed vertices in different transitive sets. So we conclude that in our case (in galaxies with spiders) allowing middle stars and centers of stars to be between leaves of
another star, will require looking for leaves of the same star in different transitive sets. And so we don't have necessarily the correct type of adjacency between the leaves. Hence the method used in galaxies will not succeed in our case. Now roughly speaking, in constellations we have only right and left stars, and the centers of stars are allowed to be between leaves of stars as long as we are able to look for the center of a star in a transitive set that is completely different that that used to look for its leaves (see \cite{kg}), which makes the proof of constellations work. To sum up the ideas of the proof in constellations we proceed first as in the galaxy case, but instead of looking for centers in linear sets, we look for centers in transitive sets also. Then we start constructing the stars star by star. If the star is found we move to the second, and if not we will get two transitive sets, such that one of them is complete to the other one. We repeat this procedure several times, so that we can construct a sequence of transitive sets such that the one appearing first is complete to the transitive sets following it. So we use this sequence of transitive sets and enrich them with linear sets to proceed again as in galaxies and complete the proof for constellations.   So if we tried to use the technique of constellations for galaxies with spiders we will be obliged to look for a center of the star and its leaves in the same
transitive set. This is obviously impossible because no backward arcs can be found in a transitive set. This implies that we can't proceed as in constellations and look for centers of stars also in transitive sets. So the paradigm followed in constellations is not useful in our case.\vspace{2mm}

As we said at the begining we will proceed as in galaxies only by looking for the centers of the stars in linear sets and looking for leaves in transitive sets, but instead of looking for leaves in the same transitive set, we will look for them in different transitive sets when needed. What is different now is that we don't have the right type of adjacency between the leaves. That is we don't necessarily get as an outcome a spider. Instead, we will get a triangle. Here is the place where we need to think about another orderings of the vertices of our galaxy with spiders and change   the way of thinking: "\textit{searching for our galaxy with spiders by following its nebula ordering}". To this end we form the set $\Theta_{\theta}(H)$ that consists of crucial orderings of $H$ and that are obtained from $\theta$ by performing permutation for the enterior vertices of some spiders of $H$ under $\theta$. What is important about these orderings is that the portion where we have middle spiders or centers of stars between leaves of another stars is trasformed to triangles. Note also instead of using corresponding structures as that defined for galaxies (see \cite{polll}), we will  start by a smooth $(c,\lambda ,w)$-structure $\chi$ with $w=(0,1,0,1,....,0,1)$ and $\mid$$w$$\mid$ large enough, then we will construct the backward arcs only for the portion where we have middle stars and centers of stars between leaves of another stars, and then according to the outcomes we will know which ordering in $\Theta_{\theta}(H)$ we will use to complete constructing $H$ and that by using one of the associated smooth structures corresponding to $H$ under $\theta$.\\
 We are ready to give an intuition how the proof works for galaxies with spiders. Let us fix some galaxy with spiders $H$ and let $\theta$ be its galaxy with spiders ordering. By assuming that the contrary is true and fixing $\epsilon >0$ small enough, we can assume that there exists a smooth $(c,\lambda ,w)$-structure $\chi$ in an $H$-free $\epsilon$-critical tournament $T$. Our goal is to build a copy of $H$ in $T$. We will choose $\chi$ to be large enough (that is $\mid$$w$$\mid$ is large enough), with $w=(0,1,0,1,...,0,1)$. Now we will extract from $\chi$ an appropriate smooth structure that is galaxy with spiders-correlated to $H$ under $\theta$, say $\widehat{\chi}$. Now using $\widehat{\chi}$ we will construct the mutant spiders of $H$ under $\theta$ one by one after updating the smooth structure in a way that we can merge the constructed mutant spider with the previous ones. Our goal now is to extract from $\Theta_{\theta}(H)$ the ordering that will save us. When all the mutant spiders denoted by $\mathsf{S}_i$ are constructed, we will be in front of two possibilities: either $T$$\mid$$V(\mathsf{S}_i)$ is a spider, or $T$$\mid$$V(\mathsf{S}_i)$ contains a triangle. If the former holds, the spider is fixed. Otherwise, the triangle found is fixed and the rest of the vertices of $V(\mathsf{S}_i)$ are ignored. We repeat this procedure for all the mutant spiders constructed (that is for all $i$). We now get a sequence of spiders and triangles constructed in $\chi$. There exists a unique ordering $\theta^{'}$ in $\Theta_{\theta}(H)$ having $\{s,t\}$-vector as the $\{s,t\}$-vector corresponding to the sequence of spiders and triangles constructed in $\chi$. This ordering $\theta^{'}$ is the hoped ordering. Now in order to complete constructing $H$, we will follow the ordering $\theta^{'}$. Besides the sets in $\chi$ used to construct the spiders and triangles, we will add some sets from $\chi$ to form a galaxy with spiders-associated smooth structure $\overline{\chi}$ of $H$ under $\theta^{'}$. We can do that because we started by a large enough smooth structure $\chi$. Note that the outcome will be one of the associated structures in $\chi_{\theta^{'}}(H)$. Finally we complete constructing the rest of the stars of $H$ under $\theta^{'}$ star by star and after merging each star by the triangles and spiders previously constructed. After merging all the stars with the spiders and triangles we get a copy of $H$ in $T$, and we are done.     
\subsection{Proof of Theorem \ref{p}} \label{proof}
\hspace{3.5mm} In the following lemma we follow the induction technique used in \cite{polll} and Lemma \ref{r} to prove that we can find a copy of the mutant spiders of a galaxy with spiders in any clutter-associated smooth structure. This is done by constructing mutant spiders one by one.  
\begin{lemma}\label{o}
Let $H$ be a clutter under an ordering $\theta$ of its vertices with $\mid$$H$$\mid$ $= h$. Let $\mathcal{S}_{1},...,\mathcal{S}_{l}$ be the spiders of $H$ under $\theta$. Let $0 < \lambda < \frac{1}{(2h)^{h+3}}$, $c > 0$ be constants, and $w$ be a $\lbrace 0,1 \rbrace$-vector. Fix $k \in \lbrace 0,...,l \rbrace$ and let $\widehat{\lambda} = (2h)^{l-k}\lambda$ and $\widehat{c} = \frac{c}{(2h)^{l-k}}$. There exist $ \epsilon_{k} > 0$ such that $\forall 0 < \epsilon < \epsilon_{k}$, for every $\epsilon$-critical tournament $T$ with $\mid$$T$$\mid$ $= n$ containing $\chi = (S_{1},...,S_{\mid w \mid})$ as a smooth $(\widehat{c},\widehat{\lambda},w)$-structure which is clutter-associated to $H^{k} = H$$\mid$$\bigcup_{j=1}^{k}V(\mathcal{S}_{j})$ (where $H^{l} = H$, and $H^{0}$ is the empty tournament) under the ordering $\theta_{k} = (h_{1},...,h_{\mid H^{k} \mid})$ (where $\theta_{k}$ is the restriction of $\theta$ to $V(H^{k})$), we have $\tilde{H}^{k}$ is well-contained in $\chi$ (where $\tilde{H}^{k}$ is the mutant clutter corresponding to $H^{k}$ under $\theta $).  
\end{lemma}
\begin{proof}
The proof is by induction on $k$. For $k=0$ the statement is obvious since $\tilde{H}^{0}$ is the empty digraph. Suppose that $\chi = (S_{1},...,S_{\mid w \mid})$ is a smooth $(\widehat{c},\widehat{\lambda},w)$-structure in $T$ which is clutter-associated to $H^{k}$ under the ordering $\theta_{k}$. Let $\mathcal{S}_{k} = \lbrace h_{q_{1}},...,h_{q_{p}} \rbrace$ and assume without loss of generality that $\mathcal{S}_{k} = \lbrace h_{q_{1}},...,h_{q_{m}},h_{q_{m+1}},...,h_{q_{m+5}},$ $h_{q_{m+6}},...,h_{q_{p}} \rbrace$ is a middle spider$_{1}$. For all $ i \in [p]$, let $D_{i} = \lbrace v \in \displaystyle{\bigcup_{j=1}^{\mid w \mid}S_{j}};$ $ \xi(v) = q_{i} \rbrace$. Then there exists $ 1 \leq y<y+4 \leq \mid$$w$$\mid$, such that $D_{i} \subseteq S^{i}_{y}$ for $i=1,...,m$, $D_{m+1} = S_{y+1}$, $D_{m+2} = S^{1}_{y+2}$, $D_{m+3} = S^{2}_{y+2}$, $D_{m+4} = S^{3}_{y+2}$, $D_{m+5} = S_{y+3}$, $D_{i} \subseteq S_{y+4}$ for $i=m+6,...,p$, with $w(y)=w(y+2)=w(y+4)= 1$ and $w(y+1) =w(y+3)= 0$. Since we can assume that $\epsilon < log_{\frac{\widehat{c}}{2h}}(1-\frac{\widehat{c}}{h})$, then by Lemma \ref{r}, there exist vertices $d_{1},...,d_{m},d_{m+1},d_{m+4}$, such that $d_{i} \in D_{i}$ for $i=1,...,m+1,m+4$ and $\lbrace d_{1},...,d_{m}\rbrace \leftarrow d_{m+1} \leftarrow d_{m+4}$.\\ Let $D_{m+2}^{*} = \lbrace d_{m+2}\in D_{m+2}; \lbrace d_{1},...,d_{m+1}\rbrace\rightarrow d_{m+2} \rbrace$, $D_{i}^{*} = \lbrace d_{i}\in D_{i}; \lbrace d_{1},...,d_{m+1},d_{m+4}\rbrace\rightarrow d_{i} \rbrace$ for $i=m+5,...,p$. Then by Lemma \ref{g}, $\mid$$D_{m+5}^{*}$$\mid$ $\geq (1-h\widehat{\lambda})\widehat{c}n\geq \frac{\widehat{c}}{2}n$ since $\widehat{\lambda} \leq \frac{1}{2h}$ and $\mid$$D_{i}^{*}$$\mid$ $\geq \frac{1-h^{2}\widehat{\lambda}}{h}\widehat{c}tr(T)\geq \frac{\widehat{c}}{2h}tr(T)$ for $i=m+2,m+6,...,p$ since $\widehat{\lambda} \leq \frac{1}{2h^{2}}$. Since we can assume that $\epsilon < log_{\frac{\widehat{c}}{4h}}(1-\frac{\widehat{c}}{2h})$, then by Lemma \ref{r}, there exist vertices $d_{m+2},d_{m+5},...,d_{p}$, such that $d_{i} \in D^{*}_{i}$ for $i=m+2,m+5,...,p$ and $d_{m+2}\leftarrow d_{m+5}\leftarrow\lbrace d_{m+6},...,d_{p}\rbrace $. 
  Let $D_{m+3}^{*} = \lbrace d_{m+3}\in D_{m+3}; \lbrace d_{1},...,d_{m+1}\rbrace\rightarrow d_{m+3} \rightarrow \lbrace d_{m+5},...,d_{p}\rbrace \rbrace$. Then by Lemma \ref{g}, $\mid$$D_{m+3}^{*}$$\mid$ $\geq (1-3h\widehat{\lambda})\frac{\mid S_{y+2}\mid}{3}\geq \frac{\mid S_{y+2}\mid}{6} $ since $\widehat{\lambda} \leq \frac{1}{6h}$. Then $D_{m+3}^{*} \neq\phi$. Fix $d_{m+3}\in D_{m+3}^{*}$. So $T$$\mid$$\lbrace d_{1},...d_{p}\rbrace$ contains a copy of $\tilde{H}^{k}$$\mid$$V(\tilde{S_{k}})$. Denote this copy by $Y$.\\
  For all $ i \in \lbrace 1,...,\mid$$w$$\mid \rbrace \backslash \lbrace y,...,y+4 \rbrace$, let $S_{i}^{*} = \displaystyle{\bigcap_{x\in V(Y)}S_{i,x}}$.  Then by Lemma \ref{g}, $\mid$$S_{i}^{*}$$\mid$ $\geq (1-h\widehat{\lambda})\mid$$S_{i}$$\mid$ $\geq \frac{1}{2h}\mid$$S_{i}$$\mid$ since $\widehat{\lambda} \leq \frac{2h-1}{2h^{2}}$.
Write $\mathcal{H} = \lbrace 1,...,\mid $$H^{k}$$ \mid \rbrace \backslash \lbrace q_{1},...,q_{p} \rbrace$. If $\lbrace v\in S_{y}: \xi(v) \in \mathcal{H} \rbrace \neq \phi$, then define $J_{y} = \lbrace \eta \in \mathcal{H}: \exists v \in S_{y}$ and $\xi(v)= \eta \rbrace$. Now for all $ \eta \in J_{y}$, let $S_{y}^{*\eta}= \lbrace v \in S_{y}: \xi(v)=\eta$ and $v \in\bigcap_{i=m+1}^{p} S_{y,d_{i}} \rbrace$. By Lemma \ref{g}, for all $ \eta \in J_{y}$, we have $\mid$$S_{y}^{*\eta}$$\mid$ $\geq (1-h\delta^{w}(y)\widehat{\lambda})\frac{\mid S_{y}\mid}{\delta^{w}(y)} \geq \frac{1-h^{2}\widehat{\lambda}}{h}\mid $$S_{y}$$\mid $ $\geq \frac{\mid S_{y}\mid}{2h}$ since $\widehat{\lambda} \leq \frac{1}{2h^{2}}$. Now for all $ \eta \in J_{y}$, select arbitrary $\lceil \frac{\mid S_{y}\mid}{2h}\rceil$ vertices of $S_{y}^{*\eta}$ and denote the union of these $\mid$$J_{y}$$\mid$ sets by $S_{y}^{*}$.
So we have defined $t$ sets $S_{1}^{*},...,S^{*}_{t}$, where $t = \mid$$w$$\mid$ $-4$ if $S_{y}^{*}$ is defined and $t = \mid$$w$$\mid$ $-5$ if $S_{y}^{*}$ is not defined. We have $\mid$$S_{i}^{*}$$\mid$ $\geq \frac{\widehat{c}}{2h}tr(T)$ for every defined $S_{i}^{*}$ with $w(i) = 1$, and $\mid$$S_{i}^{*}$$\mid$ $\geq \frac{\widehat{c}}{2h}n$ for every defined $S_{i}^{*}$ with $w(i) = 0$. Now Lemma \ref{b} implies that $\chi^{*}=(S_{1}^{*},...,S^{*}_{t})$ form a smooth $(\frac{\widehat{c}}{2h},2h\widehat{\lambda},w^{*})$$-$structure of $T$ which is clutter-associated to $\tilde{H}^{k-1}$ under $\theta_{k-1}=(h_{1},...,h_{\mid H^{k-1} \mid})$ for an appropriate $\lbrace 0,1 \rbrace$$-$vector $w^{*}$. 
Now take $\epsilon_{k} < min \lbrace \epsilon_{k-1}, log_{\frac{\widehat{c}}{4h}}(1-\frac{\widehat{c}}{2h}) \rbrace$. So by induction hypothesis $\tilde{H}^{k-1}$ is well-contained in $\chi^{*}$. Now by merging the well-contained copy of $\tilde{H}^{k-1}$ and $Y$  we get a copy of $\tilde{H}^{k}$. $\hfill{\square}$ 
\end{proof}
\vspace{4mm}

In \cite{polll} Berger et al. proved that if $H$ is a galaxy under $\theta$, then we can find $H$ as a well-contained copy in a corresponding structure of $H$ under $\theta$. Following a similar proof, in the following lemma we will prove that we can find $H$ as a well-contained copy in any  galaxy-associated smooth structure of $H$  under $\theta$ (note that the corresponding structure is one of the galaxy-associated structures).  
\begin{lemma} \label{l}
Let $H$ be a regular galaxy under an ordering $\theta$ of its vertices with $\mid$$H$$\mid$ $= h$. Let $Q_{1},...,Q_{m}$ be the stars of $H$ under $\theta$. Let $0 < \lambda < \frac{1}{(2h)^{h+2}}$, $c > 0$ be constants, and $w$ be a $\lbrace 0,1 \rbrace$$-$vector. Fix $k \in \lbrace 0,...,m \rbrace$ and let $\widehat{\lambda} = (2h)^{m-k}\lambda$ and $\widehat{c} = \frac{c}{(2h)^{m-k}}$. There exist $ \epsilon_{k} > 0$ such that $\forall 0 < \epsilon < \epsilon_{k}$, for every $\epsilon$$-$critical tournament $T$ with $\mid$$T$$\mid$ $= n$ containing $\chi = (S_{1},...,S_{\mid w \mid})$ as a smooth $(\widehat{c},\widehat{\lambda},w)$$-$structure which is galaxy-associated to $H^{k}= H$$\mid$$\bigcup_{j=1}^{k}V(Q_{j})$ (where $H^{m} = H$ and $H^{0}$ is the empty tournament) under the ordering $\theta_{k}= (h_{1},...,h_{\mid H^{k} \mid})$, we have $H^{k}$ is well-contained in $\chi$.  
\end{lemma}
\begin{proof}
The proof is by induction on $k$. For $k=0$ the statement is obvious since $H^{0}$ is the empty tournament. Suppose that $\chi = (S_{1},...,S_{\mid w \mid})$ is a smooth $(\widehat{c},\widehat{\lambda},w)$$-$structure in $T$ which is galaxy-associated to $H^{k}$ under the ordering $\theta_{k}$.  Let $h_{p_{0}}$ be the center of $Q_{k}$ and $h_{p_{1}},...,h_{p_{q}}$ be its leaves for some integer $q>0$. 
For all $ i \in \{0,...,q\}$, let $R_{i} = \lbrace v \in \displaystyle{\bigcup_{j=1}^{\mid w \mid}S_{j}};$ $ \xi(v) = p_{i} \rbrace$.
Then there exists $ z \in \lbrace 1,...,\mid$$w$$\mid \rbrace$ with $w(z)=0$, and there exists $ f \in \lbrace 1,...,\mid$$w$$\mid \rbrace$ with $w(f)=1$, such that $R_{0} = S_{z}$ and for all $ i \in [q]$, $R_{i} \subseteq S_{f}$. Since we can assume that $\epsilon < log_{\frac{\widehat{c}}{2h}}(1-\frac{\widehat{c}}{h})$, then by Lemma \ref{r}, there exist vertices $r_{0},r_{1},...,r_{q}$, such that $r_{i} \in R_{i}$ for $i=0,1,...,q$ and \\
$\ast$ $r_{1},...,r_{q}$ are all adjacent from $r_{0}$ if $z>f$.\\
$\ast$ $r_{1},...,r_{q}$ are all adjacent to $r_{0}$ if $z<f$.\\
So $T$$\mid$$\lbrace r_{0},r_{1},...,r_{q} \rbrace$ induces a copy of $H^{k}$$\mid$$V(Q_{k})$. Denote this copy by $Y$.\\    
For all $ i \in \lbrace 1,...,\mid$$w$$\mid \rbrace \backslash \lbrace z,f \rbrace$, let $S_{i}^{*} = \displaystyle{\bigcap_{p\in V(Y)}S_{i,p}}$.  Then by Lemma \ref{g}, $\mid$$S_{i}^{*}$$\mid$ $\geq (1-\mid$$Y$$\mid\widehat{\lambda})\mid$$S_{i}$$\mid$ $\geq (1-h\widehat{\lambda})\mid$$S_{i}$$\mid$ $\geq \frac{1}{2h}\mid$$S_{i}$$\mid$ since $\widehat{\lambda} \leq \frac{2h-1}{2h^{2}}$.
Write $\mathcal{H} = \lbrace 1,...,\mid $$H^{k}$$ \mid \rbrace \backslash \lbrace p_{0},...,p_{q} \rbrace$. If $\lbrace v\in S_{f}: \xi(v) \in \mathcal{H} \rbrace \neq \phi$, then define $J_{f} = \lbrace \eta \in \mathcal{H}: \exists v \in S_{f}$ and $\xi(v)= \eta \rbrace$. Now for all $ \eta \in J_{f}$, let $S_{f}^{*\eta}= \lbrace v \in S_{f}: \xi(v)=\eta$ and $v \in S_{f,r_{0}} \rbrace$. Then by Lemma \ref{g}, for all $ \eta \in J_{f}$, we have $\mid$$S_{f}^{*\eta}$$\mid$ $\geq (1-\delta^{w}(f)\widehat{\lambda})\frac{\mid S_{f}\mid}{\delta^{w}(f)} \geq \frac{1-h\widehat{\lambda}}{h}\mid $$S_{f}$$\mid $ $\geq \frac{\mid S_{f}\mid}{2h}$ since $\widehat{\lambda} \leq \frac{1}{2h}$. Now for all $ \eta \in J_{f}$, select arbitrary $\lceil \frac{\mid S_{f}\mid}{2h}\rceil$ vertices of $S_{f}^{*\eta}$, and denote the union of these $\mid$$J_{f}$$\mid$ sets by $S_{f}^{*}$. 
So we have defined $t$ sets $S_{1}^{*},...,S^{*}_{t}$, where $t = \mid$$w$$\mid$ $-1$ if $S_{f}^{*}$ is defined and $t = \mid$$w$$\mid$ $-2$ if $S_{f}^{*}$ is not defined. We have $\mid$$S_{i}^{*}$$\mid$ $\geq \frac{\widehat{c}}{2h}tr(T)$ for every defined $S_{i}^{*}$ with $w(i) = 1$, and $\mid$$S_{i}^{*}$$\mid$ $\geq \frac{\widehat{c}}{2h}n$ for every defined $S_{i}^{*}$ with $w(i) = 0$. Now Lemma \ref{b} implies that $\chi^{*}=(S_{1}^{*},...,S^{*}_{t})$ form a smooth $(\frac{\widehat{c}}{2h},2h\widehat{\lambda},w^{*})$$-$structure of $T$ which galaxy-associated to $H^{k-1}$, where $\frac{\widehat{c}}{2h}= \frac{c}{(2h)^{m-(k-1)}}, 2h\widehat{\lambda}=(2h)^{m-(k-1)}\lambda$, and $w^{*}$ is an appropriate $\lbrace 0,1 \rbrace$-vector.
Now take $\epsilon_{k} < min \lbrace \epsilon_{k-1}, log_{\frac{\widehat{c}}{2h}}(1-\frac{\widehat{c}}{h}) \rbrace$. So by induction hypothesis $H^{k-1}$ is well-contained in $\chi^{*}$. Now by merging the well-contained copy of $H^{k-1}$ and $Y$ we get a copy of $H^{k}$. $\hfill{\square}$ 
\end{proof}
\vspace{4mm}

We are ready to prove our main theorem in this section: \vspace{4mm}\\
\noindent \sl {Proof of Theorem \ref{p}. }\upshape
Fix some galaxy with spiders $H$ and let $\theta = (v_{1},...,v_{h})$ be its galaxy with spiders ordering, with $\mid$$H$$\mid$ $=h$. Let $d\in \mathbb{N}^{*}$ be an appropriate integer ($d= (20h^{5})!$). Let $\mathcal{S}_{1},...,\mathcal{S}_{l}$  be the spiders of $H$ under $\theta$. Let $K= H$$\mid$$\bigcup_{i=1}^{l}V(\mathcal{S}_{i})$ and let $\tilde{\theta}=(v_{s_{1}},...,v_{s_{r}})$ be the restriction of $\theta$ to $V(K)$. Let $\tilde{K}$ be the mutant clutter corresponding to $K$ under $\tilde{\theta}$. Let $T$ be an $\epsilon$-critical tournament. It is enough to show that for $\epsilon > 0$ small enough $T$ is not $H$-free. By Lemma \ref{e}, we may assume that $\mid$$T$$\mid$ is large enough. We may assume by Corollary \ref{c} that $T$ contains a smooth $(c_{0},\lambda_{0},w)$-structure $\chi =(A_{1},...,A_{\mid w\mid})$, for some appropriate constant $c_{0}>0$, arbitrary small enough positive number $\lambda_{0}$, and $\lbrace 0,1 \rbrace$-vector $w=(0,1,0,1,...,0,1)$, where $\mid$$w$$\mid$ is large enough, such that for all $ i\in [\mid$$w$$\mid ]$, $\mid$$A_{i}$$\mid$ is divisible by $d$. $A_{1},...,A_{\mid w\mid}$ are called \textit{sets of} $\chi$. Let $i_{s}$ be such that $w(i_{s}) = 1$. Assume that $A_{i_{s}} = \lbrace a^{1}_{i_{s}},...,a_{i_{s}}^{\mid A_{i_{s}} \mid} \rbrace$ and $(a^{1}_{i_{s}},...,a_{i_{s}}^{\mid A_{i_{s}} \mid})$ is a transitive ordering. Write $m(i_{s}) = \frac{\mid A_{i_{s}} \mid}{d}$. Denote $A^{j}_{i_{s}} = \lbrace a^{(j-1)m(i_{s})+1}_{i_{s}},...,a_{i_{s}}^{jm(i_{s})} \rbrace$ for $j \in \lbrace 1,...,d \rbrace$. Note that the sets $A^{j}_{i_{s}}$ for $j=1,...,d$ are also called \textit{sets of} $\chi$. We may assume that we can extract from $\chi$ an appropriate smooth $(\frac{c_{0}}{d},d\lambda_{0},\hat{w})$-structure $\hat{\chi}=(F_{1},...,F_{\mid \hat{w}\mid})$ which is galaxy with spiders-correlated to $H$ under $\theta$, where $\hat{w}$ is an appropriate $\lbrace 0,1 \rbrace$-vector. The extracted smooth $(\frac{c_{0}}{d},d\lambda_{0},\hat{w})$-structure  $\hat{\chi}$ satisfies the following: Let $i\in [\mid$$\hat{w}$$\mid ]$. If $\hat{w}(i)=0$, then $F_{i}=A_{i_{1}}$ with $ i_{1}\in$ $ [\mid$$w$$\mid ]$ and $w(i_{1})=0$. If $\hat{w}(i)=1$, then $F_{i}=\displaystyle{\bigcup_{j\in B}}A_{i_{t}}^{j}$ $=\bigcup_{j=1}^{\delta^{\hat{w}}(i)}F_{i}^{j}$ with $w(i_{t})=1$, appropriate $B=\lbrace j_{1},...,j_{\delta^{\hat{w}}(i)}\rbrace\subseteq \lbrace 2h,...,d-2h\rbrace$ such that for all $ 1\leq f\leq \delta^{\hat{w}}(i)-1$, $j_{f+1}-j_{f}>2h$ (note that for all $ f\in\lbrace 1,...,\delta^{\hat{w}}(i)-1\rbrace$, $F_{i}^{f}$ is complete to $F_{i}^{j}$ for $j = f+1,...,\delta^{\hat{w}}(i)$). Moreover, for all $ i\in [\mid$$\hat{w}$$\mid -1]$ with $F_{i}\subseteq A_{r_{1}}$ and $F_{i+1}\subseteq A_{r_{2}}$ for some $1\leq r_{1}<r_{2}\leq \mid$$w$$\mid$, we have: $r_{2}-r_{1}> 2h$ and $r_1>h$.\\ Now for all $ i \in [r]$, let  $S_{i} = \lbrace v \in \bigcup_{j=1}^{\mid \hat{w} \mid}F_{j};$ $ \xi(v) = s_{i} \rbrace$. For all $  j \in $ $[\mid$$\hat{w}$$\mid ]$, let $F^{*}_{j}=  \displaystyle{\bigcup_{S_{i}\subseteq F_{j}}S_{i}}$ (notice that $F^{*}_{j}\subseteq F_{j}$ or $F^{*}_{j}= \phi$). Let
 $F^{*}_{1},...,F^{*}_{\mid w^{*}\mid}$ denote the nonempty sets $F^{*}_{j}$. Then $\chi^{*}=(F^{*}_{1},...,F^{*}_{\mid w^{*}\mid})$ is a smooth $(\frac{c_{0}}{d},d\lambda_{0},w^{*})$-structure which is clutter-associated to $K$ under $\tilde{\theta}$. Since we can assume that $\lambda_{0} < \frac{1}{d(2h)^{h+3}}$, then we can use Lemma \ref{o} and conclude (taking $\epsilon$ small enough and $k=l$) that $\tilde{K}$ is well-contained in $\chi^{*}$. Denote by $\mathsf{S}_{i}$ the well-contained copy of $\tilde{\mathcal{S}_{i}}$ in $T$, and let $\gamma_{i}$ be the ordering of $V(\mathsf{S}_{i})$ according to their appearance in $\chi^{*}$ for $i=1,...,l$. Notice that we don't know the orientation of the arcs in $E(T$$\mid$$V(\mathsf{S}_{i}))\backslash E(\mathsf{S}_{i})$. Let $ i\in [l]$. If all the arcs in $E(T$$\mid$$V(\mathsf{S}_{i}))\backslash E(\mathsf{S}_{i})$ are forward with respect to $\gamma_{i}$, then $T$$\mid$$V(\mathsf{S}_{i})$ induces a copy of $\mathcal{S}_{i}$ in $T$, and in this case we redenote this copy by $X_{i}$. Otherwise a triangle will appear in $T$, that is there exists an arc $e_{i}\in E(T$$\mid$$V(\mathsf{S}_{i}))\backslash E(\mathsf{S}_{i})$, such that $e_{i}$ is backward with respect to $\gamma_{i}$. Let $e_{i}=(f_{i},a_{i})$ and let $y_{i}\in V(\mathsf{S}_{i})$, such that $T$$\mid$$\lbrace y_{i},f_{i},a_{i}\rbrace$ is a triangle of $T$$\mid$$V(\mathsf{S}_{i})$ under $\gamma_{i}$. Denote this triangle by $C_{i}$. Now apply the above rule for all $i\in [l]$. We get a sequence of spiders and triangles. Let $\lbrace \mathcal{K}_{1},...,\mathcal{K}_{l}\rbrace$ denote this sequence, where either $\mathcal{K}_{i}=X_{i}$ or $\mathcal{K}_{i}=C_{i}$. Let $G=T$$\mid$$\bigcup_{i=1}^{l}V(\mathcal{K}_{i})$ and let $\tilde{\beta}$ be the ordering of the vertices of $G$ according to their appearance in $\chi^{*}$. So there exist $\theta^{'}\in \Theta_{\theta}(H)$ such that $g^{H,\theta^{'}}=g^{G,\tilde{\beta}}$. Let $\mathsf{K}_{1},...,\mathsf{K}_{l}$ be the spiders and triangles of $H$ under $\theta^{'}=(u_{1},...,u_{h})$. Let $H_{\mathsf{K}} = H$$\mid$$\bigcup_{j=1}^{l}V(\mathsf{K}_{j})$ and let $\theta^{'}_{\mathsf{K}}=(u_{q_{1}},...,u_{q_{p}})$ be the restriction of $\theta^{'}$ to $\bigcup_{j=1}^{l}V(\mathsf{K}_{j})$. Write $\tilde{\beta} =(b_{1},...,b_{p})$ (we only rename the vertices in the ordering $\tilde{\beta}$ as $b_{1},...,b_{p})$. \\
Now let $i \in [\mid$$w^{*}$$\mid ]$ such that $w^{*}(i)=0$. If $x\notin F^{*}_{i}$ for all $x\in V(G)$, then we remove $F^{*}_{i}$ from $\chi^{*}$. Let $ i \in [\mid$$w^{*}$$\mid ]$ such that $w^{*}(i)=1$, and let $1\leq k\leq\delta^{w^{*}}(i)$. If $x\notin F^{*k}_{i}$ for all $x\in V(G)$, then we remove $F^{*k}_{i}$ from $\chi^{*}$.  We do this for all $1\leq k\leq\delta^{w^{*}}(i)$. We apply this for all $ i \in [\mid$$w^{*}$$\mid ]$. We get from $\chi^{*}$ a new smooth $(\frac{c_{0}}{d},d\lambda_{0},\tilde{w})$-structure $\widetilde{\chi^{*}}$ for some appropriate $\lbrace 0,1\rbrace$-vector $\tilde{w}$. Now since $\mid$$w$$\mid$ is large enough then we can assume that we can add some sets to $\widetilde{\chi^{*}}$ from $\chi$ to form a  smooth $(\frac{c_{0}}{d},d\lambda_{0},z)$-structure $\overline{\chi}= (E_{1},...,E_{\mid z\mid})$ which is galaxy with spiders-associated to $H$ under $\theta^{'}$ for some appropriate $\lbrace 0,1\rbrace$-vector $z$, such that $b_{i}\in B_{i}=\lbrace v \in \bigcup_{j=1}^{\mid z \mid}E_{j};$ $ \xi(v) = q_{i} \rbrace$ for $i=1,...,p$. Let $H_{Q}=H$$\mid$$(V(H)\backslash V(H_{\mathsf{K}}))$ and let $\theta^{'}_{Q}=(u_{p_{1}},...,u_{p_{q}})$ be the restriction of $\theta^{'}$ to $V(H_{Q})$. Let $Q_{1},...,Q_{m}$ be the stars of $H_{Q}$ under $\theta_{Q}^{'}$. We may assume that $H_{Q}$ is a regular galaxy under $\theta^{'}_{Q}$. For all $ i \in [q]$, let $R_{i}^{*}=\displaystyle{\bigcap_{x\in V(G)}}R_{i,x}$, where $R_{i} = \lbrace v \in \bigcup_{j=1}^{\mid z \mid}E_{j};$ $ \xi(v) = p_{i} \rbrace$. Let $  i \in [q]$. If $R_{i}= E_{j_{1}}$ for some $ j_{1} \in $ $[\mid$$z$$\mid ]$ with $z(j_{1})=0$, then by Lemma \ref{g}, $\mid$$R_{i}^{*}$$\mid$ $\geq (1-hd\lambda_{0})\mid$$E_{j_{1}}$$\mid \geq \frac{\mid E_{j_{1}}\mid}{2} \geq \frac{\mid E_{j_{1}}\mid}{2h}$ since we can assume that $\lambda_{0}\leq\frac{1}{2hd}$. In this case we only rename the set $R_{i}^{*}$ by $R^{**}_{i}$. Let $  i \in [q]$. If $R_{i}\subseteq E_{j_{2}}$ for some $ j_{2} \in $ $[\mid$$z$$\mid ]$ with $z(j_{2})=1$, then by Lemma \ref{g}, $\mid$$R_{i}^{*}$$\mid$ $\geq (1-h^{2}d\lambda_{0})\frac{\mid E_{j_{2}}\mid}{h}  \geq \frac{\mid E_{j_{2}}\mid}{2h}$ since we can assume that $\lambda_{0}\leq\frac{1}{2h^{2}d}$. In this case we select arbitrary $\lceil \frac{\mid E_{j_{2}}\mid}{2h}\rceil$ vertices from $R_{i}^{*}$ and we denote by $R_{i}^{**}$ the set of the selected $\lceil \frac{\mid E_{j_{2}}\mid}{2h}\rceil$ vertices. Now for all $ j \in $ $[\mid$$z$$\mid ]$, let $\overline{E}_{j}=  \displaystyle{\bigcup_{R^{**}_{i}\subseteq E_{j}}R^{**}_{i}}$ (notice that $\overline{E}_{j}\subseteq E_{j}$ or $\overline{E}_{j}= \phi$). Let $\overline{E}_{1},...,\overline{E}_{\mid \overline{z}\mid}$ denote the nonempty sets $\overline{E}_{j}$. Also notice that for all $ j\in [\mid$$\overline{z}$$\mid ]$, $\mid$$\overline{E_{j}}$$\mid$ $\geq \frac{\mid E_{s}\mid}{2h}$ for some $ s\in$ $[\mid$$z$$\mid ]$. Then $\chi^{'} = (\overline{E}_{1},...,\overline{E}_{\mid \overline{z}\mid})$ form a smooth $(\frac{c_{0}}{2hd},2hd\lambda_{0},\overline{z})$-structure of $T$ which is galaxy-associated to $H_{Q}$ under $\theta^{'}_{Q}$ for an appropriate $\lbrace 0,1 \rbrace$-vector $\overline{z}$. Now since we can assume that $\lambda_{0} < \frac{1}{d(2h)^{h+3}}$, then we can use Lemma \ref{l} and conclude (taking $\epsilon$ small enough and $k=m$) that $H_{Q}$ is well-contained in $\chi^{'}$. Denote this copy by $W$. Now by merging $W$ and $G$ we get a copy of $H$ in $T$. So $T$ contains $H$ as a subtournament. That completes the proof of Theorem \ref{p}. $\hfill{ \square}$

\section{Merging Operation}
  A tournament $T$ is obtained by \textit{merging an asterism and a galaxy with spiders} if there exist an ordering $ \theta $ of its vertices such that $V(T)$ is the disjoint union of $V(\mathcal{A}^{\beta}_{1}),...,V(\mathcal{A}^{\beta}_{l}),Z,X$ where $\mathcal{A}^{\beta}_{1},...,\mathcal{A}^{\beta}_{l}$ are the $\beta$$-$asteroids of $T$ under $\theta$, $T$$\mid$$(Z\cup X)$ is a galaxy with spiders under $\hat{\theta}$ where $\hat{\theta}$ is the restriction of $\theta$ to $Z\cup X$, $T$$\mid$$Z$ is a clutter under $\theta_{Z}$ where $\theta_{Z}$ is the restriction of $\theta$ to $Z$, and $T$$\mid$$X$ is a galaxy under $\theta_{X}$ where $\theta_{X}$ is the restriction of $\theta$ to $X$, no vertex of a $\beta$$-$asteroid appears in the ordering $\theta$ between leaves of a star of $T$$\mid$$X$ under $\theta_{X}$, no vertex of a $\beta$$-$asteroid appears in the ordering $\theta$ between legs of a spider of $T$ under $\theta$ incident to the same center of this spider, and $\theta$ satisfies the following: let $v$ and $v'$ be legs of distinct spiders $\mathcal{S}_{i}$ and $\mathcal{S}_{j}$ of $T$$\mid$$(Z\cup X)$ under $\hat{\theta}$, and suppose that $v$ is in some $M\in P_{\hat{\theta}}(T$$\mid$$(Z\cup X))$. If there is a vertex of a $\beta$$-$asteroid of $T$ under $\theta$ lying between a center of $\mathcal{S}_{i}$ and a center of $\mathcal{S}_{j}$ under $\theta$ then $v' \notin M$.\vspace{2mm}
  
  By this merging procedure we can build infinitely many tournaments satisfying $EHC$ from asterisms and galaxies with spiders. 
  \begin{theorem}
  Every tournament obtained by merging an asterism and a galaxy with spiders satisfies EHC.
  \end{theorem}
We omit the proof of the above theorem since it is a combination between the proof of the Theorems \ref{j} and \ref{p}.

\end{document}